\documentclass[11pt]{amsart}

\usepackage{amsmath, amsthm, amsfonts, amssymb, mathrsfs, amscd}
\usepackage{graphicx}
\usepackage{dsfont}
\usepackage{tikz-cd}
\usepackage[mathcal]{euscript}
\usepackage{subcaption}
\usepackage{pgfplots}
\usepackage{csquotes}
\usepackage{subfiles}
\usepackage[english]{babel}
\usepackage{url}
\usepackage[
    backend=biber, 
    style=alphabetic,
    doi=false,      
    url=false,      
    isbn=false,     
    eprint=true    
]{biblatex}
\DeclareFieldFormat[article,inbook,incollection,inproceedings]{title}{#1}
\usepackage[colorlinks=true, 
            linkcolor=blue,    
            citecolor=blue,    
            urlcolor=blue]{hyperref}

\newcommand{\C}{\mathbb{C}}
\newcommand{\R}{\mathbb{R}}

\newcommand{\D}{\mathcal{D}}

\newcommand{\Q}{\mathbb{Q}}
\newcommand{\LL}{\mathcal{L}}
\renewcommand{\P}{\mathbb{P}}

\renewcommand{\O}{\mathcal{O}}
\newcommand{\F}{\mathcal{F}}
\newcommand{\I}{\mathcal{I}}
\newcommand{\V}{\mathcal{V}}

\newcommand{\NN}{\mathcal{N}}
\newcommand{\TT}{\mathbb{T}}
\newcommand{\m}{\mathfrak{m}}

\newcommand{\sm}{\mathrm{sm}}

\newcommand{\red}{\mathrm{red}}

\newcommand{\W}{\mathcal{W}}
\renewcommand{\L}{\mathbb{L}}

\newcommand{\A}{\mathbf{A}}
\renewcommand{\AA}{\mathbb{A}}

\newcommand{\nd}{\mathrm{nd}}
\newcommand{\U}{\mathcal{U}}
\newcommand{\M}{\mathcal{M}}

\newcommand{\E}{\mathcal{E}}
\newcommand{\da}{\dagger}
\newcommand{\CC}{\mathcal{C}}

\newcommand{\emp}{\emptyset}

\newcommand{\Rsim}{\sim_R}
\newcommand{\realplaces}{\Omega_{k,\mathrm{real}}}
\newcommand{\RatCurves}{\operatorname{RatCurves}}
\newcommand{\Sec}{\operatorname{Sec}}
\renewcommand{\Im}{\operatorname{Im}}

\newcommand{\res}{\operatorname{Res}}

\newcommand{\Sch}{\operatorname{Sch}}
\newcommand{\Hom}{\operatorname{Hom}}
\newcommand{\Hilb}{\operatorname{Hilb}}

\newcommand{\Gal}{\operatorname{Gal}}
\newcommand{\Proj}{\operatorname{Proj}}
\newcommand{\codim}{\operatorname{codim}}
\newcommand{\degp}{\operatorname{deg}}

\newcommand{\charf}{\operatorname{char}}
\newcommand{\Spec}{\operatorname{Spec}}
\newcommand{\fin}{\operatorname{fin}}

\newcommand{\Ext}{\operatorname{Ext}}
\newcommand{\AP}{\operatorname{AP}}
\DeclareMathOperator{\id}{id}
\DeclareMathOperator{\ev}{ev}

\calclayout

\newtheorem{thm}{Theorem}[section]

\newtheorem{cor}[thm]{Corollary}
\newtheorem{conv}[thm]{Convention}
\newtheorem{prop}[thm]{Proposition}
\newtheorem{lem}[thm]{Lemma}

\newtheorem{quest}[thm]{Question}
\newtheorem{defp}[thm]{Definition and Proposition}

\theoremstyle{definition}
\newtheorem{defn}[thm]{Definition}

\theoremstyle{remark}
\newtheorem{rem}[thm]{Remark}

\pgfplotsset{compat=1.9}
\title[Arithmetic purity for complete intersections]{Arithmetic purity of strong approximation for complete intersections of large dimension}

\author[C. Zhang]{Chen Zhang}
\address{School of Mathematical Sciences, Capital Normal University, 105 Xisanhuanbeilu, 100048 Beijing, China}
\email{chenzhangwtp@gmail.com}

\begin{document}

\begin{abstract}
    In this paper, we prove that for a smooth complete intersection $X\subset \P^n$ of type $(d_1,\dots,d_c)$ over a number field $k$, there exists an effective bound $N_{\AP}(d_1,\dots,d_c)$ such that if $n\ge N_{\AP}(d_1,\dots,d_c)$, then $X$ satisfies arithmetic purity of strong approximation. We establish a connecting lemma that promotes weak approximation to strong approximation under a suitable connectivity hypothesis via rational curves over $k$. We construct the required connecting conics by combining the deformation theory of rational curves with results from the circle method.
\end{abstract}
\maketitle

\tableofcontents
\addtocontents{toc}{\protect\setcounter{tocdepth}{1}}

\section{Introduction}

\subsection{Arithmetic Purity for Strong Approximation}
Let $k$ be a number field. It is well known that weak approximation is a $k$-birational invariant property of smooth geometrically integral varieties, while strong approximation fails to be.
However, due to various geometric and cohomological purity results, one can expect that strong approximation satisfies purity. This question was proposed by Wittenberg in \cite[Question 2.11]{Wit15}.

\begin{quest}\label{wit conj}
    Let $X$ be a smooth geometrically integral variety over a number field $k$ that satisfies strong approximation off $S\subset \Omega_k$, and let $Z$ be any codimension-2 closed subvariety of $X$. Does $X\setminus Z$ satisfy strong approximation off $S\subset \Omega_k$?
\end{quest}

We call this property the arithmetic purity of strong approximation.
The first example of arithmetic purity for $\mathbb{A}^n$ off any $S\neq \emptyset$ was observed by Cao and Xu in \cite{CaoXu2018} and Wei in \cite{Wei2021}. Then the arithmetic purity of quasi-split semisimple simply connected algebraic groups was proved by Cao, Liang and Xu in \cite{CaoLiangXu2019} and extended to isotropic groups by Cao and Huang in \cite{CaoHuang2020}.
The conjecture about arithmetic purity with the Brauer-Manin obstruction was also proposed and confirmed for the corresponding homogeneous spaces in \cite{CaoLiangXu2019}. In \cite{Wei2021}, Wei also proved the arithmetic purity with the Brauer-Manin obstruction for toric varieties. 
In a recent work \cite{cao2025effectiveequidistributionarithmeticpurity}, Cao, Huang and Zhang further confirmed the arithmetic purity with the Brauer-Manin obstruction for affine quadric hypersurfaces off any isotropic place $v$. Since weak approximation is equivalent to strong approximation for smooth proper varieties, it is natural to expect arithmetic purity for strong approximation to hold off $\emptyset$ as well. In \cite{chen2024arithmetic}, Chen refined the arithmetic purity with the Brauer-Manin obstruction for certain complete toric varieties off $\emptyset$.

The existing approaches to algebraic groups, homogeneous spaces, toric varieties, and quadrics rely largely on their special equivariant structures, through subgroup actions and quotient constructions, universal torsors and descent. By contrast, our method relies on rational curves. The key idea is to use sufficiently flexible families of rational curves to bridge weak approximation and strong approximation. For smooth complete intersections of sufficiently large dimension, we realize this approach through conics passing through two rational points, combining the deformation theory of rational curves with results from the circle method.

Let $X\subset \mathbb P^n_k$ be a smooth complete intersection of type $(d_1,\ldots,d_c)$ over
a number field $k$. Our main result proves that $X$ satisfies arithmetic purity of strong approximation for each fixed $(d_1,\dots,d_c)$ and sufficiently large $n$.

\begin{thm}[Theorem \ref{APSA uneffective}]\label{intro APSA}
Let $X\subset \mathbb P^n_k$ be a smooth complete intersection of type $(d_1,\ldots,d_c)$ over a number field $k$. There exists an effective bound $N_{\mathrm{AP}}(d_1,\ldots,d_c)$ such that, if
\[
n\geq N_{\mathrm{AP}}(d_1,\ldots,d_c),
\]
then $X$ satisfies arithmetic purity of strong approximation.
\end{thm}

The bounds obtained by the general argument are effective but are not intended to be optimal. In low degree one can use sharper local and circle-method estimates. For instance, a smooth cubic hypersurface of dimension at least $38$ satisfies arithmetic purity of strong approximation. Further explicit bounds are recorded in Corollary \ref{table of NAP}.

\subsection{Proof Strategy}

The basic arithmetic tool of our approach is the connecting lemma \ref{connecting lemma int qua}. In a simplified form, it says that if a smooth variety $U$ satisfies weak approximation and suitable pairs of $k$-rational points of $U$ can be connected by $\P^1_k$, then $U$ satisfies strong approximation. 
Since weak approximation for $U=X\setminus Z$ is known by the circle method under our dimension assumptions, the connecting lemma reduces the problem to constructing suitable connecting rational curves in $U$.

To this end, we study the moduli spaces $M_{x,x'}$ of conics on $X$ passing through two prescribed rational points $x,x'\in X(k)$. For a general pair $(x,x')$, a theorem of Pan \cite{Pan2018} shows, under a suitable dimension hypothesis, that $M_{x,x'}$ is a smooth complete intersection in $\mathbb P^{n-2}$ of explicitly known multidegree.
This makes it possible to apply results from the circle method directly to the two-point conic space. 
Our use of the circle method on spaces of rational curves is inspired by the work of Browning-Vishe and Browning-Sawin \cite{Browning_2017,Browning2018FreeRC}.

There are two further geometric issues. First, we need smooth conics which avoid the codimension-two subset $Z$. We introduce $1$-free conics following \cite{Kollar1996} and deform them to avoid the $2$-codimensional obstruction $Z$. 
Also, we give a criterion for the existence of $1$-free smooth conics by applying generic smoothness to the evaluation morphism and using a smoothing process. Consequently, we can show that 1-free smooth conics in $M_{x,x'}$ form an open nonempty subset for a general pair $(x,x')\in X\times X$.

Second, at real places we require real $1$-free smooth conics to connect suitable pairs of real points, and these conics correspond to the real solutions of $M_{x,x'}$.  
For a real point $x\in X(\R)$, we will show how to find a complex line $L$ passing through $x$ such that the reducible conic $C_0=L\cup \overline{L}$ is $1$-free, and then smooth it into a family of real $1$-free smooth conics using a smoothing lemma in \cite{Kollar1999}.

For finite places, the $p$-adic solutions of $M_{x,x'}$ follow in sufficiently large dimension from results on systems of forms such as \cite{Wooley1998}. From these local solutions, the circle method, in the form of Frei-Madritsch \cite{Frei_2016}, then gives that
\[
M_{x,x'} \text{ has Zariski dense $k$-rational points}.
\]
Since the locus of $1$-free smooth conics avoiding $Z$ is a nonempty open subset of $M_{x,x'}$, it contains a $k$-rational point. The corresponding conic $C$ contains $x\in X(k)$ and is therefore isomorphic to $\mathbb P^1_k$; it gives the connecting curve required by the connecting lemma. Strong approximation for $U$ follows.

\subsection{Rational Curves as an Arithmetic Method}

The strategy outlined above fits into a broader tradition of using rational curves to investigate arithmetic properties of algebraic varieties.

Over function fields of complex curves, Koll\'{a}r--Miyaoka--Mori used the deformation theory of free rational curves to establish a zeroth-order form of weak approximation for rationally connected varieties, assuming the existence of a rational point \cite{KMM92}. Graber--Harris--Starr subsequently proved the existence of rational points for such varieties, and de Jong--Starr extended this existence result to separably rationally connected varieties in arbitrary characteristic \cite{GraberHarrisStarr2003,deJongStarr2003}.

In the setting of function fields of complex curves, Hassett--Tschinkel \cite{HT06} built on these results to establish weak approximation at places of good reduction for rationally connected varieties by allowing prescribed higher-order jets. For open varieties, Chen--Zhu \cite{chen2015strongapproximationfunctionfields} used logarithmic $\mathbb A^1$-curves to establish strong approximation results for certain smooth affine complete intersections, and Boughattas \cite{Bou24} further studied arithmetic purity of strong approximation in this setting.

Arithmetic over number fields is more subtle, since rational connectedness alone does not guarantee the existence of rational points.

Classical secant-tangent constructions, originating in the work of Segre and developed systematically by Manin \cite{Manin1986}, provide fundamental tools for the study of rational points and rational curves on cubic hypersurfaces, with applications to unirationality and $R$-equivalence.
Swinnerton-Dyer further refined and iterated the Segre-Manin secant-tangent constructions to prove weak weak approximation for smooth cubic surfaces over number fields with a rational point \cite{SwinnertonDyer01}. 
More recently, Tian-Zhang \cite{TianZhang18} established weak approximation for smooth cubic hypersurfaces of dimension at least $10$ over purely imaginary number fields by constructing $k$-rational curves through suitable pairs of conjugate points.

These developments indicate a general principle: sufficiently flexible families of rational curves can serve not only to construct rational points, but also to impose prescribed local conditions on them. 
The present work continues this geometric approach by using $k$-rational curves to study strong approximation over number fields.

\subsection{An application to $R$-equivalence}
Our method of connecting rational points by rational curves, developed for strong approximation, also applies to the study of $R$-equivalence introduced in \cite{Manin1968,Manin1986}.
Recall that two rational points on a smooth proper variety $X$ are said to be directly $R$-equivalent if they lie on the image of a $k$-morphism from $\mathbb P^1_k$ to $X$. The equivalence relation generated by direct $R$-equivalence is called $R$-equivalence, and the set of equivalence classes is denoted by $X(k)/R$.

One might expect that $X(k)/R$ is finite for a rationally connected variety $X$ over a number field $k$.
This question was proposed by Colliot-Th\'el\`ene:

\begin{quest}\cite[Question~10.12]{CT2011}
  Let $X$ be a smooth, projective, geometrically rationally connected variety over a number field $k$. Is $X(k)/R$ finite?
\end{quest}
Some related works were summarized in \cite[\S~2.4]{Wit15}.

Our construction of rational curves shows that general rational points on smooth complete intersections of large dimension are contained in finitely many $R$-equivalence classes.
\begin{thm}\label{intro R finiteness}[Theorem \ref{R finiteness}]
Let $X\subset \mathbb P^n_k$ be a smooth complete intersection of type $(d_1,\ldots,d_c)$ over a number field $k$. There exists an effective constant $N_R(d_1,\ldots,d_c)$ such that, if
\[
n\geq N_R(d_1,\ldots,d_c),
\]
then there exists a Zariski open dense subvariety $X^\circ\subset X$ such that
\[
\operatorname{Im} [X^\circ(k)\longrightarrow X(k)/R] \text{ is finite.}
\]
\end{thm}

\subsection{The Organization of the Paper}

The paper is organized as follows.
Sections~2 and~3 develop the deformation theory, freeness, and codimension-two avoidance needed for conics.
Section~4 recalls Pan's theorem and the relevant results on spaces of lines and conics on complete intersections. 
Section~5 treats smoothing, in particular the real-place construction. 
Section~6 combines these geometric results with local solubility, the circle method, and the connecting lemma to prove the main Theorem~\ref{intro APSA} and to record explicit bounds in several low-degree cases.
Section~7 translates the connecting curves we constructed into $R$-equivalence.
\bigskip

\subsection*{Acknowledgements}
I am grateful to Professor Fei Xu for introducing me to the arithmetic purity problem and for his guidance, and to Professor Sheng Chen for suggesting the problem of arithmetic purity for intersections of quadrics.

During the preparation of this manuscript, I used ChatGPT to assist with literature searches and synthesis, English-language editing, the checking of mathematical arguments, and the exposition of selected auxiliary mathematical arguments, including parts of Appendices A–C. I independently checked the relevant references and mathematical arguments and take full responsibility for the mathematical results and the content of the paper.

\section{Moduli Spaces and Deformation}

In this section, we review some basic theory of the Hilbert scheme and the deformation theory of curves. Our results are based on \cite{Kollar1996}. In Lemma \ref{smooth locus}, we relate the locus where the differential of the evaluation map is surjective to the freeness of the curve.

\begin{conv}
    Throughout this section, $X\subset \P^n$ is a smooth projective variety over a field $k$ with $\O_X(1)=\O_{\P^n}(1)|_X$.
\end{conv}

First, we fix some notation.
\begin{defn}\label{defn of M U}
    (i) Let $M$ be the Hilbert scheme of $X$ over $k$ with Hilbert polynomial $P\in \Q[t]$. We always let $M$ parameterize curves $C$, i.e., $\deg P=1$.
    \[
    M=\Hilb_{P}(X/k).
    \]

    (ii) The Hilbert scheme has the universal family $\U$ equipped with natural projections $p:\U\to M$ and $q:\U \to X$, which is described as 
    \[
    \U=\{(C,x)\in M\times X\mid x\in C\}.
    \]

    (iii) The two-pointed incidence space $\V=\U\times_M \U$, equipped with natural projections $p_2:\V\to M$ and $q_2:\V \to X\times X$, can be described as
    \[
    \V=\{(C,x,x')\in M\times X\times X \mid x,x'\in C \}.
    \]

    (iv) Let $\V^\circ\subset \V$ be the open locus parametrizing curves with two distinct marked points, i.e.,
    \[
    \V^\circ=\{(C,x,x')\in \V \mid x\neq x'  \}.
    \]

    (v) Let $M_{x}$ and $M_{x,x'}$ be the closed subschemes of $M$ parameterizing curves passing through $x$, and through $x,x'$, respectively. They are defined as the fibers $q^{-1}(x)$ and $q_2^{-1}(x,x')$, respectively.
\end{defn}

We are interested in smooth curves $C\subset X$, so we introduce the moduli space parameterizing smooth curves.

\begin{defn}\label{defn of good curve}
    (i) Let $M_{\nd}\subset M$ be the open subscheme parameterizing those smooth curves in $X$, i.e.,
    \[
    M_{\nd}=\{ [C]\in M\mid C\cong p^{-1}([C]) \text{ is smooth}\} .
    \]
    That $M_{\nd}$ is an open locus comes from the fact that it is the locus where the fiber of the finitely presented, proper, flat projection $p$ is smooth. See \cite[Theorem 12.2.4]{GrothendieckDieudonne1966EGAIV3}.
   
    (ii) Let the open subschemes $\U_{\nd}\subset \U$ , $\V_{\nd}\subset \V$ and $\V_{\nd}^\circ\subset \V_{\nd}$ be the incidence spaces over $M_{\nd}$ with one marked point, two marked points, and two distinct marked points, respectively (see definition \ref{defn of M U}). They can be defined as the pullbacks of $M_{\nd}\subset M$ along the projections $p$ and $p_2$. The natural projections $p_{\nd}$, $q_{\nd}$ are well defined.
\end{defn}

\begin{lem}\label{smoothness of Cxx'}
    (i) Let $u=(C,x)\in \U$. If $[C]\in M$ is a smooth point and $x$ is a smooth point of $C$, then $u$ is a smooth point of $\U$.
    
    (ii) Let $v=(C,x,x')\in \V$. If $[C]\in M$ is a smooth point and $x,x'$ are smooth points of $C$, then $v$ is a smooth point of $\V$.
\end{lem}

\begin{proof}
    Recall that $p:\U \to M$ is a flat family by the universal property of the Hilbert scheme $M$.
    
    For $u=(C,x)$ and $[C]=p(u)$, we have $p^{-1}([C])\cong C$ by definition of $\U$. Then $x\in C_{\sm}$ by assumption, so the morphism $p$ is smooth at $u$ by the fiberwise criterion of smoothness (see \cite{StacksProject-01V9}). Now $M\to \Spec k$ is smooth at $[C]$ by assumption, so the composite morphism $\U \xrightarrow{p} M \rightarrow \mathrm{Spec}k $ is smooth at the point $u$. This completes the proof of part (i).

    Note that $p_2:\V \to M$ is a flat family and $p_2^{-1}([C])\cong C\times C$ for $v=([C],x,x')\in \V$. The same argument proves part (ii).
\end{proof}
 The following lemma is a corollary of the basic description $T_{[C]}M \cong H^0(C, \NN_{C/X})$ of the tangent space for the Hilbert scheme. A very general version describing the tangent spaces of the flag Hilbert schemes (in particular, the universal family of the Hilbert scheme) was written in \cite[Proposition 4.5.3]{Sernesi2006}.

\begin{lem}\label{tangent space of good U V}
(i) Let $[C]\in M$ be a point. The tangent space of $M$ at $[C]$ can be described as
\[
T_{[C]}M \cong H^0(C, \NN_{C/X}).
\]

(ii) Let $u=(C,x)\in \U$ with $x\in C_{\sm}$. The tangent space of $\U$ at $u$ can be described as
\[
T_{(C,x)}\U \cong \left\{ (\xi, v_x) \in H^0(C, \NN_{C/X}) \oplus T_x X \mid \xi(x) + \bar{v}_x=0 \right\},
\]
where $\xi(x)$ and $\overline{v}_x$ are the images of $\xi$ and $v_x$ under the evaluation map to $\NN_{C/X}|_x$, respectively.

(iii) Let $v=(C,x,x')\in \V$ with $x,x'\in C_\sm$. The tangent space of $\V$ at $v$ can be described as
\[
T_{(C,x,x')}\V \cong \left\{ (\xi, v_x, v_{x'}) \in H^0(C, \NN_{C/X}) \oplus T_x X \oplus T_{x'} X \mid  \xi(x) +\bar{v}_x=0 \text{ and } \xi(x') + \bar{v}_{x'}=0 \right\},
\]
where $\xi(x)$, $\overline{v}_x$, $\xi(x')$, and $\overline{v}_{x'}$ are the images of $\xi$, $v_x$, $\xi$, and $v_{x'}$ under evaluation maps to $\NN_{C/X}|_x$ and $\NN_{C/X}|_{x'}$, respectively.

(iv) Let $x\in X$ and $C\in M_x$ satisfy $x\in C_\sm$. Then
\[
T_{[C]} M_x \cong H^0(C, \NN_{C/X}\otimes \I_x).
\]
\end{lem}

\begin{proof}
    Part (i) is the fundamental result of deformation theory; see \cite[I, Theorem 2.8]{Kollar1996}.

    Part (ii) comes from the tangent-normal sequence of the closed immersion $i:\U \to M\times X$. 
    
Let $u=(C,x)\in \U$ with $x\in C_\sm$. We know $\U\to M$ is smooth at the point $(C,x)$ by the fiberwise criterion of smoothness, since $x\in C_\sm$, and $M\times X\to M$ is smooth since $X$ is smooth over $k$.

    Now $\U \to M\times X$ is a closed immersion over the base $M$, so it is a regular closed immersion at the point $u=(C,x)$ (i.e., $\I_{\U,u}\subset \O_{M\times X,u}$ is generated by a regular sequence) by \cite[Lemma 31.23.11]{StacksProject_063P}.
    
    We have the following tangent-normal sequence as $k$-vector spaces at the point $u=(C,x)$ which is left exact. 
 \begin{equation}\label{normal tangent seq}
    0\longrightarrow T_{u}\U\longrightarrow T_{[C]}M \oplus T_x X  \xrightarrow {\delta_u} \NN_{\U/M\times X}|_u
   \tag{A}
    \end{equation}
    where $\NN_{\U/M\times X}|_u\cong (\I_{\U}/\I_{\U}^2|_u)^\vee$ because $\U\to M\times X$ is a regular embedding at $u=(C,x)$ and hence $\I_{\U}/\I_{\U}^2$ is locally free at $u$ by \cite{StacksProject_0638}.

Let $m=[C]\in M$ and $u=(C,x)\in \U$ with $x\in C_\sm$. We have the following commutative diagram with closed immersions and projections where the two rectangles on the right-hand side are Cartesian.
\begin{equation}\label{diagram U MtimesX}\tag{B}
\begin{CD}
\{(m,x)\} @>{}>> \{m\}\times C @>{j_C}>> \U\\
@V{}VV @V{i_C}VV @V{i_\U}VV\\
\{(m,x)\} @>{}>> \{m\}\times X @>{j_X}>> M\times X\\
@V{}VV @V{}VV @V{p_M}VV\\
\{m\}     @>{}>> \{m\}  @>{j_m}>> M
\end{CD}
\end{equation}

Let $\I_\U\subset\O_{M\times X}$ and $\I_C\subset\O_X$ be the ideal sheaves of $\U$ and $C$, respectively. Since $\U\to M$ is flat, $(i_\U)_*\O_\U$ is flat over $M$. Thus, the exact sequence
\[
0\longrightarrow\I_\U\longrightarrow\O_{M\times X}\longrightarrow(i_\U)_*\O_\U\longrightarrow0
\]
remains exact after pulling back along $j_X$, or equivalently, after tensoring with $\kappa(m)$ over $\O_{M,m}$. Hence we have
\[
0\longrightarrow j_X^*\I_\U\longrightarrow\O_X\longrightarrow(i_C)_*\O_C\longrightarrow0.
\]
Consequently, $j_X^*\I_\U\cong\I_C$. Pulling back this isomorphism along $i_C$ and using the commutative diagram \eqref{diagram U MtimesX}, we obtain
\[
j_C^*(\I_\U/\I_\U^2)\cong\I_C/\I_C^2.
\]
Since $\I_\U/\I_\U^2$ is locally free at $u=(C,x)$, taking duals gives
\begin{equation}\label{normal is N_C/X(x)}
(j_C^*\NN_{\U/M\times X})_x\cong(\NN_{C/X})_x.
\tag{C}
\end{equation}

The diagram \eqref{diagram U MtimesX} and the identification \eqref{normal is N_C/X(x)} then give
\begin{equation}\label{descrip of N_C/X}
\NN_{\U/M\times X}|_u\cong (j_C^*\NN_{\U/M\times X})|_x\cong \NN_{C/X}|_x.
\tag{D}
\end{equation}
Now $T_{[C]}M=H^0(C,\NN_{C/X})$ by part (i) of this lemma. Combining the description \eqref{descrip of N_C/X} with the sequence \eqref{normal tangent seq}, we obtain 
\begin{equation}\label{tangent U sequence}
  0\longrightarrow  T_{(C,x)}\U \longrightarrow  H^0(C,\NN_{C/X}) \oplus T_x X \xrightarrow {\delta_{C,x}} \NN_{C/X}|_x.
  \tag{E}
\end{equation}
Then part (ii) of the lemma follows by reading $T_{(C,x)}\U$ as the kernel of $\delta_{C,x}$ in the sequence \eqref{tangent U sequence}.

Part (iii) of the lemma can be derived from the normal sequence of the closed immersion $i:\V \to M\times X\times X$ similarly. By the same argument, we have the following exact sequence at $v=(C,x,x')\in \V$ with $x,x'\in C_\sm$:
\begin{equation}\label{tangent V sequence}
  0\longrightarrow  T_{(C,x,x')}\V \longrightarrow  H^0(C,\NN_{C/X}) \oplus T_x X \oplus T_{x'}X \xrightarrow {\delta_{C,x,x'}} \NN_{C/X}|_x\oplus \NN_{C/X}|_{x'} .
  \tag{F}
\end{equation}
Then part (iii) of the lemma follows by reading $T_{(C,x,x')}\V$ as the kernel of $\delta_{C,x,x'}$ in the sequence \eqref{tangent V sequence}.

(iv) Since $p^{-1}([C])\cong C$ and $x\in C_\sm$, we know $p:\U\to M$ is smooth at $u=(C,x)$ with fiber $C$, and combining the tangent-normal sequence of $C\hookrightarrow X$ at $x$ gives the commutative diagram
\[
\begin{CD}\label{CD of Mx}
0 @>>> T_xC @>>> T_u\mathcal U @>>> H^0(C,\NN_{C/X}) @>>> 0\\
@. @| @V{dq_u}VV @V{\operatorname{ev}_x}VV @.\\
0 @>>> T_xC @>>> T_xX @>>> \NN_{C/X}|_x @>>> 0
\end{CD}
\tag{G}
\] 
where the right vertical arrow is evaluation at $x$.

Since $M_x\cong q^{-1}(x)$, we have
\[
T_{[C]} M_x \cong T_u(q^{-1}(x))
= \ker (dq_u).
\]
Applying the snake lemma to \eqref{CD of Mx}, we obtain
\[
T_{[C]} M_x
\cong
 \ker(H^0(C,\NN_{C/X})\to \NN_{C/X}|_x)
 \cong H^0(C,\NN_{C/X}\otimes \I_x).
\]
This completes the proof of part (iv).
\end{proof}

The following lemma is similar to \cite[II.Proposition 3.5]{Kollar1996} (stated for the Hom space) in the setting of the Hilbert scheme.
\begin{lem}\label{smooth locus}
    (i) The differential $dq: T_u\U \to T_x X$ of the evaluation map $q:\U \to X$ is surjective at $u=(C,x)\in \U$ with $x\in C_\sm(k)$ if
    \[
    H^1(C,\NN_{C/X}\otimes \I_x)=0.
    \]
    If $C\cong \P^1$, the converse holds.

    (ii) The differential $dq_2: T_v\V \to T_x X\oplus T_{x'}X$ of the evaluation map $q_2:\V \to X\times X$ is surjective at $v=(C,x,x')\in \V^\circ$ with $x,x'\in C_\sm(k)$  if
    \[
    H^1(C,\NN_{C/X}\otimes \I_{x,x'})=0.
    \]
    If $C\cong \P^1$, the converse holds.

    (iii) In the case where $M=\Hilb_{2t+1}(X)$, $C= L\cup L'$ is a reducible conic consisting of a pair of lines with a unique intersection point $s$, and $v=(C,x,x')\in \V^\circ$ with $x\in L(k)\setminus\{s\}, x'\in L'(k)\setminus\{s\}$, the converse part of statement (ii) holds.
\end{lem}

\begin{proof}
The proof of part (i) of the lemma is similar to the proof of part (ii), so we only give the proofs of parts (ii) and (iii) here. 

Consider the differential of the evaluation map $q_2:\V \to X\times X$ at $q_2(v)=(x,x')$ for $v\in \V^\circ$ with $x,x'\in C_\sm$:
\[
dq_{2,(C,x,x')}: T_{(C,x,x')}\V\to T_xX \oplus T_{x'}X.
\]
By Lemma \ref{tangent space of good U V}, we can describe $T_{(C,x,x')}\V$ as
\[
T_{(C,x,x')}\V \cong \left\{ (\xi, v_x, v_{x'}) \in H^0(C, \NN_{C/X}) \oplus T_x X \oplus T_{x'} X \mid  \xi(x) +\bar{v}_x=0 \text{ and } \xi(x') + \bar{v}_{x'}=0 \right\}
\]
where $\xi(x)$, $\overline{v}_x$, $\xi(x')$, and $\overline{v}_{x'}$ are the images of $\xi$, $v_x$, $\xi$, and $v_{x'}$ under evaluation maps to $\NN_{C/X}|_x$ and $\NN_{C/X}|_{x'}$, respectively.

Thus, the differential map $dq_{2,(C,x,x')}$ is surjective if and only if for each pair 
$(\overline{v}_x,\overline{v}_{x'})\in \NN_{C/X}|_x\oplus \NN_{C/X}|_{x'}$ with
 $(v_x,v_{x'})\in T_xX\oplus T_{x'}X$, there exists a global section $\xi\in H^0(C,\NN_{C/X})$ evaluating to $-\overline{v}_x,-\overline{v}_{x'}$ at $x, x'$, respectively.

 Now $T_{x}X\oplus T_{x'}X\to \NN_{C/X}|_x\oplus \NN_{C/X}|_{x'}$ is surjective by the normal sequence of $i:C\times C\to X\times X$, since $C$ and $X$ are both smooth at $x,x'$. 
 
 Hence the differential map $dq_{2,(C,x,x')}$ is surjective if and only if the evaluation map
\begin{equation}\label{if only if normal surjective}
H^0(C,\NN_{C/X})\to \NN_{C/X}|_x\oplus \NN_{C/X}|_{x'}
\tag{A}
\end{equation}
is surjective.

Since $v=(C,x,x')\in \V^\circ$, we have $x\neq x'$. Consider the exact sequence associated to the closed immersion $\{x,x'\}\hookrightarrow C$:
\[
    0\longrightarrow \I_{x,x'} \longrightarrow \O_C \longrightarrow k(x)\oplus k(x')\longrightarrow 0.
\]
Tensoring with the vector bundle $\NN_{C/X}$ gives
\begin{equation}\label{normal seq Ixx'NN}
    0\longrightarrow \I_{x,x'}\otimes \NN_{C/X} \longrightarrow \NN_{C/X}\longrightarrow
     \NN_{C/X}|_x\oplus \NN_{C/X}|_{x'}\longrightarrow 0
\tag{B}
\end{equation}
which is exact at $x,x'$ since $\NN_{C/X}$ is locally free at $x,x'$, and exact outside $x,x'$ since the sheaf $k(x)\oplus k(x')$ vanishes outside the points $x,x'$.

So $H^0(C,\NN_{C/X})\to \NN_{C/X}|_x\oplus \NN_{C/X}|_{x'}$ is surjective if
\[
H^1(C,\I_{x,x'}\otimes \NN_{C/X})=0.
\]
Now we prove the converse statement when $C\cong \P^1$. Then the vector bundle on $\P^1$ splits by the Birkhoff-Grothendieck theorem as 
\[
\NN_{C/X} \cong  \oplus_i \O(a_i).
\]
Then $H^0(C,\NN_{C/X})\to \NN_{C/X}|_x\oplus \NN_{C/X}|_{x'}$ being surjective implies that $a_i\ge 1$, and then $H^1(C,\NN_{C/X}\otimes \I_{x,x'})=0$ as $\I_{x,x'}\cong \O_{\P^1}(-2)$. This completes the proof of part (ii).

(iii) By the argument for \eqref{if only if normal surjective}, we know that the differential $dq_2: T_v\V \to T_x X\oplus T_{x'}X$ of the evaluation map $q_2:\V \to X\times X$ is surjective at $v=(C,x,x')\in \V^\circ$ with $x,x'\in C_\sm$ if and only if the map
\[
H^0(C,\NN_{C/X})\to \NN_{C/X}|_x\oplus \NN_{C/X}|_{x'}
\]
is surjective. We also have the exact sequence of sheaves \eqref{normal seq Ixx'NN}:
\[
    0\longrightarrow \I_{x,x'}\otimes \NN_{C/X} \longrightarrow \NN_{C/X}\longrightarrow 
    \NN_{C/X}|_x\oplus \NN_{C/X}|_{x'}\longrightarrow 0.
\]
When $C=L\cup L'$ is a reducible conic, it is a hypersurface in $\P^2$ and therefore a local complete intersection. Now $X$ is a smooth variety, so $C\subset X$ is a regular immersion by \cite[Lemma 31.23.10]{StacksProject_063P}. Hence the normal sheaf $\NN_{C/X}$ is a vector bundle. Taking the long exact sequence, we find that
\begin{equation}\label{H1 cong equiv to tang surj}
    H^0(C,\NN_{C/X})\to \NN_{C/X}|_x\oplus \NN_{C/X}|_{x'} \text{ surjective}\Longleftrightarrow H^1(C,\I_{x,x'}\otimes \NN_{C/X})\cong H^1(C, \NN_{C/X}).
    \tag{C}
\end{equation}
Now, consider the normalization $L \coprod L'$ of $C=L\cup L'$ and the exact sequence associated to this normalization (see \cite{StacksProject_0C44})
\[
    0\longrightarrow \O_C \longrightarrow i_*(\O_L \oplus \O_{L'}) \longrightarrow k(s)\longrightarrow 0.
\]
Tensoring the sequence with the vector bundle $\NN_{C/X}$, we find 
\begin{equation}\label{normalization seq normal seq}
    0\longrightarrow \NN_{C/X} \longrightarrow i_*(\NN_{C/X}|_L \oplus \NN_{C/X}|_{L'}) \longrightarrow \NN_{C/X}|_s \longrightarrow 0.
    \tag{D}
\end{equation}
Now $H^0(C,\NN_{C/X})\twoheadrightarrow\NN_{C/X}|_x\oplus \NN_{C/X}|_{x'}$ is surjective with $x\in L,x'\in L'$ by \eqref{if only if normal surjective} and the assumption that the differential map is surjective. Since the evaluation map from the global sections of $\NN_{C/X}$ to $x,x'$ factors through the global sections restricted to $L$, $L'$, we have surjections
\[
H^0(L,\NN_{C/X}|_L) \twoheadrightarrow \NN_{C/X}|_x ,\;\; H^0(L',\NN_{C/X}|_{L'}) \twoheadrightarrow \NN_{C/X}|_{x'}.
\]
As $L\cong \P^1$, the vector bundles $\NN_{C/X}|_L$ and $\NN_{C/X}|_{L'}$ must split as $\bigoplus \O_{\P^1}(a_i)$ with $a_i\ge 0$ by the surjections above.

Now reading the long exact sequence associated to the normalization sequence \eqref{normalization seq normal seq}, since $H^0(L,\NN_{C/X}|_L)\twoheadrightarrow \NN_{C/X}|_s$ and $H^1(\P^1,\O(a_i))=0$ for any $a_i\ge0$, we have $H^1(C,\NN_{C/X})=0$. Combining \eqref{H1 cong equiv to tang surj}, we have
\[
H^1(C,\I_{x,x'}\otimes \NN_{C/X})\cong H^1(C, \NN_{C/X})=0.
\]
This completes the proof of part (iii).
\end{proof}

\begin{conv}
    Throughout this paper, we fix $\O_X(1)=\O_{\P^n}(1)|_X$ on $X\subset \P^n$ and set $\O_C(1)=\O_X(1)|_{C}$ for any curve $C\subset X$. 
\end{conv}

\begin{defp}\label{free is open} For an integer $\varrho\in \mathbb{Z}$, the locus of normally $\varrho$-free curves
    \[
    M_{\nd,\varrho}=\{ [C]\in M_{\nd} \mid H^{1}(C,\NN_{C/X} \otimes \O_C(-\varrho))=0\}
    \]
    is Zariski open in $M_{\nd}$, and thus forms an open subscheme.
\end{defp}

\begin{proof}
    This is a consequence of the semicontinuity theorem for cohomology.

    Consider the projection $p_{\nd}:\U_{\nd} \to M_{\nd}$. Let $\NN=(\NN_{\U/ M\times X})|_{\U_{\nd}}$ be the restriction of the normal bundle associated to the closed immersion $\U\to M\times X$. Then $\NN$ is a vector bundle on $\U_{\nd}$ since $\U_\nd$ and $M_\nd\times X$ are smooth over $M$ (see proof of Lemma \ref{tangent space of good U V}). Furthermore, for a point $[C]\in M_{\nd}$ and $j_C: p_{\nd}^{-1}([C])\to \U_{\nd}$, we have
    \[
    j_C^*\NN=\NN_{C/X}
    \]
     by the identity \eqref{normal is N_C/X(x)} in the proof of Lemma \ref{tangent space of good U V}.
    
    Set $\F_\varrho=\NN \otimes q_{\nd}^*\O_X(-\varrho)$, which is a vector bundle over $\U_{\nd}$. For the flat proper morphism $p_{\nd}:\U_{\nd} \to M_{\nd}$, set 
    \[
    f([C])=h^1(p_{\nd}^{-1}([C]),\F_\varrho |_{p_{\nd}^{-1}([C])})=h^1(C, \NN_{C/X} \otimes \O_C(-\varrho)).
    \]
    Then the locus
    \[ 
    M_{\nd, \varrho}=\{ [C]\in M_{\nd} \mid f([C])=h^{1}(C,\NN_{C/X} \otimes \O_C(-\varrho))=0\}
    \]
    is Zariski open in $M_{\nd}$ by the semicontinuity theorem for cohomology groups (see \cite[III, Theorem 12.8]{Hartshorne1977}).
\end{proof}

We show that normally free curves are unobstructed and therefore correspond to smooth points of $M=\Hilb_{P}(X)$.
\begin{lem}\label{free curve are smooth}
    (i) Let $[C]\in M_\nd$. If $H^1(C,\NN_{C/X})=0$, then
    \[
    M \text{ is smooth at the point } [C] \text{ with } \dim_{[C]}M=\chi_C(\NN_{C/X}).
    \]
    (ii) For an integer $\varrho\ge 0$, we have
    \[
    H^1(C,\NN_{C/X}\otimes \O_C(-\varrho))=0 \Longrightarrow H^1(C,\NN_{C/X})=0.
    \]
    (iii) The normally $\varrho$-free locus $M_{\nd,\varrho}$ is a smooth open subscheme of $M_{\nd}$ for any $\varrho \ge 0$.
\end{lem}
\begin{proof}
    (i) By the theorem \cite[I.Theorem 2.8]{Kollar1996}, we have (1) $T_{[C]} M\cong H^0(C, \NN_{C/X})$, (2) the dimension of every irreducible component of $M$ at $[C]$ is at least
    \[
    \dim H^0(C,\NN_{C/X})- \dim \operatorname{Obs}(C).
    \]
    When $C\subset X$ is a regular immersion (i.e., a local complete intersection), then $\operatorname{Obs}(C)$ is contained in $H^1(C,\NN_{C/X})$ by \cite[I.Lemma 2.12]{Kollar1996} and \cite[I.Proposition 2.14]{Kollar1996}. Now $C$ and $X$ are both smooth, so $C\subset X$ is a regular immersion by \cite[Lemma 31.23.11]{StacksProject_063P}.
    
    We have $H^1(C,\NN_{C/X})=0$ by assumption, so $h^0(C, \NN_{C/X}) = \dim T_{[C]} M \ge \dim_{[C]} M\ge  h^0(C,\NN_{C/X})$. Since $\chi_C(\NN_{C/X})= h^0(C,\NN_{C/X})-h^1(C,\NN_{C/X})$ by definition,
    \[
    \dim T_{[C]} M = \dim_{[C]} M=\chi_C(\NN_{C/X}).
    \]
    Hence, $M$ is smooth at the point $[C]$ with $\dim_{[C]} M=\chi_C(\NN_{C/X})$. This completes the proof of part (i).

Now we prove part (ii). It is clear for $\varrho =0$, so we assume $\varrho>0$. Since $\O_C(1)$ is the restriction of $\O_{\P^n}(1)$, it is globally generated. Hence $\O_C(\varrho)$ is globally generated for every $\varrho> 0$.
For $\varrho>0$, choose a section $s\in H^0(C,\O_C(\varrho))$ whose restriction to each connected component of $C$ is nonzero, and let $D$ be the corresponding Cartier divisor. Multiplication by $s$ gives an exact sequence
\[
0\longrightarrow \O_C(-\varrho)\xrightarrow{\cdot s} \O_C
\longrightarrow\O_D \longrightarrow 0.
\]

Tensoring with the locally free sheaf $\NN_{C/X}$, we obtain
\[
0\longrightarrow \NN_{C/X}\otimes\O_C(-\varrho) \longrightarrow
\NN_{C/X} \longrightarrow \NN_{C/X}|_D
\longrightarrow 0.
\]
Since $D$ is zero-dimensional, $H^1(D,\NN_{C/X}|_D)=0.$

Therefore, the associated long exact sequence in cohomology yields a surjection
\[
H^1(C,\NN_{C/X}\otimes\O_C(-\varrho))
\twoheadrightarrow
H^1(C,\NN_{C/X}).
\]
This completes the proof of part (ii).

Part (iii) is a corollary of part (i) and part (ii), together with Definition \ref{free is open}.
\end{proof}

\section{Free Conics and Lines}

In this section, we define $0$-free lines and $1$-free smooth conics on a variety $X$. We give a criterion for the existence of $1$-free smooth conics using generic smoothness and the smoothing process of curves (which will be proved in Section 5); see Theorem \ref{criterion dominant for 1-free conic}.

Assuming the existence of $1$-free smooth conics, we show how to use $1$-free smooth conics to avoid a $2$-codimensional closed subvariety $Z\subset X$; see Lemma \ref{formally avoid obs Z}. This section is also based on the theory of free curves in \cite{Kollar1996}.

\begin{conv}
     Throughout this section, $X\subset \P^n$ is a smooth projective variety over a field $k$ with $\O_X(1)=\O_{\P^n}(1)|_X$.
\end{conv}

\begin{conv}
    From now on, we specialize the notation $M$, $\U$, and $\V$ introduced in Section 2 to conics as follows:

    (i) The Hilbert scheme $M$ is fixed as 
    \[
    M=\Hilb_{2t+1}(X),
    \]
    which is the moduli space of conics on $X$.
    
    (ii) The universal family $\U$, the two-pointed incidence space $\V$, and their associated open subschemes are defined as in Section 2, with $M$ parameterizing conics.
\end{conv}

\begin{defn}\label{defn Fano line F L}
    
    (i) The Fano scheme of lines 
    \[
    F=\Hilb_{t+1}(X)
    \]
     parameterizes lines in $X$.

    (ii) Let $\L$ be the universal family over the Fano scheme $F$ of lines on $X$. Let $p:\L \to F$ and $q:\L\to X$ be the natural projections. 

    (iii) Note that every line $[L]\in F$ is a smooth curve itself, so $F_\nd=F$ in the notation of Section 2. 

    (iv) Let $F_x\subset F$ be the closed subscheme parameterizing lines passing through $x\in X$, which is defined as the fiber $q^{-1}(x)$.

\end{defn}

Fix the line bundle $\O_X(1)=\O_{\P^n}(1)|_X$ on $X\subset \P^n$. For any curve $f:C\to X$, let $\O_C(1)=f^* \O_X(1)$. By Proposition \ref{free is open}, we have the following definition of the moduli spaces of 0-free lines and 1-free smooth conics as open subschemes of $F$ and $M$. We also introduce $1$-free reducible conics which will be used later (the term $1$-free for a reducible conic is not standard).
\begin{defn}\label{defn of M_f U_f}
    (i) Let $M_f\subset M$ be the open subscheme parameterizing 1-free smooth conics in $M_{\nd}$: 
    \[
    M_f:= \{[C]\in M_{\nd}\mid H^1(C,\NN_{C/X}\otimes \O_C(-1))=0 \}.
    \]

    (ii) Let $F_f\subset F$ be the open subscheme parameterizing 0-free lines in $F$:
     \[
     F_f:=\{[L]\in F \mid H^1(L,\NN_{L/X}\otimes \O_L(-1))=0 \}.
     \]

    (iii) Let $\U_f\subset\U$ be the restriction of the universal family over $M_f$, and let $\V_f\subset\V$ be the corresponding two-pointed incidence space. Let $\V_f^\circ= \V_f\cap \V_{\nd}^\circ$ be the open locus of $\V_f$ with two distinct marked points.

    (iv) Let $\L_f$ be the universal family over $F_f$ with one marked point.
    
    (v) For a reducible conic $[C]\in M$, we say that it is a $1$-free reducible conic if 
    \[
    H^1(C,\NN_{C/X}\otimes \O_C(-1))=0.
    \]
\end{defn}

\begin{rem}\label{rem of free defn}
    In the usual sense, $f:\P^1\to X$ is $r$-free if $H^1(\P^1, f^* T_X\otimes \O_{\P^1}(-1-r))=0.$ In our case, lines and smooth conics are rational over an algebraically closed field.
     For a line $L$, $\O_L(1)\cong \O_{\P^1}(1)$. For a smooth conic $C$, 
    $\O_C(1)\cong \O_{\P^1}(2)$. By computing the normal sequence, we see that our definition is equivalent to $0$-freeness for a line and $1$-freeness for a smooth conic.
\end{rem}

\begin{cor}\label{free are smooth points upon points are smooth}
    (i) The schemes $F_f$ and $\L_f$ are smooth over $k$.

    (ii) The schemes $M_f$, $\U_f$ and $\V_f$ are smooth over $k$.
\end{cor}
\begin{proof}
     The smoothness of the schemes $F_f$ and $M_f$ comes from Lemma \ref{free curve are smooth} and Definition \ref{defn of M_f U_f}. The smoothness of $\L_f$, $\U_f$, and $\V_f$ comes from the fiberwise criterion of smoothness; see Lemma \ref{smoothness of Cxx'}.
\end{proof}

\begin{lem}\label{smooth mor for U V free}
    (i) The evaluation maps $q_\U:\U \to X$ and $q_\L:\L\to X$ restrict to smooth morphisms
    \[
    q_{\U_f}:\U_f \to X, \quad q_{\L_f}:\L_f\to X.
    \]
    
    (ii) The evaluation map $q_2: \V\to X\times X$ restricts to a smooth morphism
    \[
    q_{2,f}: \V_f^\circ \to X\times X.
    \]

\end{lem}
\begin{proof}
    Smoothness is preserved under base change and $fpqc$ descent; see \cite{StacksProject_02VL}. Thus it is sufficient to prove that $q_{\L_f}$, $q_{\U_f}$, and $q_{2,f}$ are smooth at each point over an algebraically closed field $k$ by base change.

    Now we fix a smooth conic $[C]\in M_{\nd}$ and two distinct points $x,x'\in C$, i.e., fix $v=(C,x,x')\in \V^\circ_{\nd}$.
    
We have $C\cong \P^1$ since $C$ is a smooth conic and $k$ is algebraically closed. As $C\hookrightarrow X\hookrightarrow \P^n$ is a conic, we have $\O_{C}(1)\cong\O_{\P^1}(2)$ on $C\cong \P^1$. Also, $\I_{x,x'}\cong \O_{\P^1}(-2)$. Therefore, we have the following isomorphism
    \begin{equation}\label{Ixx' is OC-1}
        H^1(C,\NN_{C/X}\otimes \I_{x,x'}) \cong H^1(C, \NN_{C/X} \otimes \O_C(-1)).
        \tag{A}
    \end{equation}

Now $v\in \V_f^\circ$ implies that $H^1(C,\NN_{C/X}\otimes \I_{x,x'})=0$ by \eqref{Ixx' is OC-1}. By Lemma \ref{smooth locus}, the differential map $dq_{2,v}: T_v\V \to T_x X\oplus T_{x'}X$ at $v$ is surjective. Since $\V^\circ_f$ and $X\times X$ are smooth over $k$ (see Corollary \ref{free are smooth points upon points are smooth}), the surjectivity of the differential map implies that $q_{2,f}$ is smooth at $v$. This completes the proof of part (ii).

Now we prove part (i). For a smooth conic $C\cong \P^1$, we know 
\[
H^1(C,\NN_{C/X}\otimes \I_{x,x'})=0 \Longrightarrow H^1(C,\NN_{C/X}\otimes \I_x)=0.
\]
Then the smoothness of $q_{\U_f}$ follows by Corollary \ref{free are smooth points upon points are smooth} and Lemma \ref{smooth locus}.

For a line, we know
\[
H^1(L,\NN_{L/X}\otimes \I_{x}) \cong H^1(L, \NN_{L/X} \otimes \O_L(-1)).
\]
Hence, the combination of Lemma \ref{smooth locus} and Corollary \ref{free are smooth points upon points are smooth} implies the smoothness of $q_{\L_f}$ by the same argument.
\end{proof}

The following criterion for the existence of a $1$-free smooth conic is proved by an argument similar to that of \cite[II Theorem 3.11]{Kollar1996}. However, we must tackle reducible conics, which are $1$-free by the surjectivity of the tangent map, and we can smooth a $1$-free reducible conic into a smooth one.

\begin{thm}\label{criterion dominant for 1-free conic}
    Let $X\subset \P^n$ be a smooth projective geometrically integral variety over a field $k$ with $\charf k=0$. If the projection
    \[
     q_2: \V \longrightarrow X\times X
    \]
    is a dominant map, then there exists a $1$-free smooth conic on $X$, i.e., $M_f \neq \emp$.
\end{thm}

\begin{proof}
    For schemes and morphisms over $k$, the properties that $Y$ is nonempty and that $f:Y\to Z$ is dominant are preserved under arbitrary field extensions and descend along field extensions; see \cite{StacksProject_02YJ} (nonemptiness follows because $Y_K\to Y$ is surjective). So we work over the algebraically closed field $\overline{k}$ by base change. 

    When $X\subset \P^n$ is linear, i.e., $X\cong \P^m$, then every smooth conic $C$ in $X$ is $1$-free by computing $\NN_{C/X}$ directly. So we assume $X$ is not linear, and then the following open subvariety
    \[
    W=\{(x_1,x_2)\in X\times X\mid x_1\neq x_2, \text{ line } L^*=\langle x_1,x_2\rangle \not\subset X\}
    \]
    is open and nonempty, and hence dense since $X$ is geometrically integral. Since $q_2$ is dominant, we know the restriction
    \[
    q_{2,W}: \V \cap q_2^{-1}(W)\longrightarrow W
    \]
    is dominant. We denote $\W=\V \cap q_2^{-1}(W)$, which is open in $\V$.
    
    For any $(C,x,x')\in \W$, we classify the conic $C$. (1) $C$ is a smooth conic. (2) $C=L\cup L'$ is a reducible conic consisting of a pair of lines with a unique intersection point $s$. Now $x,x'\in C$ and $(x,x')\in W$, so $L^*=\langle x,x'\rangle \not\subset X$. Therefore it is impossible that the two points $x,x'$ lie on a common line $L$ or $L'$; in particular, $x,x'\neq s$. We may assume $x\in L\setminus\{s\}$, $x'\in L'\setminus\{s\}$. (3) $C$ is non-reduced, i.e., $C$ is a double line $2L^*$ passing through $x,x'$; this is impossible since $L^*\not\subset X$ by assumption. 

    Now we set $\W_{\red}$ to be the associated reduced closed subscheme of $\W$. Then $\W_{\red}$ is generically smooth since the local ring of $\W_\red$ at each generic point is a field. Let $\W'\subset \W_{\red}$ be the smooth locus, which is open and dense; we know 
    \[
    \W' \rightarrow X\times X
    \]
    is a dominant morphism of smooth varieties. Therefore we can apply generic smoothness (as $\charf k=0$) and derive a point $v=(C,x,x')\in \W_{\red}$ such that 
    \[
    dq_{2,\red,v}:T_{(C,x,x')} \W_\red \longrightarrow T_xX\oplus T_{x'}X \text{ is surjective.}
    \]
    Now the tangent map $dq_{2,v}$ factors as $T_v \W_\red\to T_v \V \to T_xX\oplus T_{x'}X$, so
    \[
     dq_{2,v}:T_{(C,x,x')} \V \longrightarrow T_x X\oplus T_{x'}X \text{ is surjective.}
    \]
    Since $v=(C,x,x')$ has $(x,x')\in W$, we know $x\neq x'$ and then $v\in \V^\circ$ by definition. By our arguments above, we know either $C$ is a smooth conic, or $C=L\cup L'$ is a reducible conic consisting of a pair of lines with a unique intersection point $s$ and $x\in L\setminus \{s\},x'\in L'\setminus\{s\}$. Applying Lemma \ref{smooth locus}, we know 
    \begin{equation}\label{H^1 NIxx' vanish}
    H^1(C, \NN_{C/X}\otimes \I_{x,x'})=0
    \tag{A}
    \end{equation}
    in either case. Suppose that $C$ is reducible; the line $L^*=\langle x,x'\rangle$ cuts $C$ along $\{x,x'\}$. Let $\I_{L^*}=\O_{\P^2}(-1)$ in $\P^2$, which is isomorphic to the pullback of $\O_{\P^n}(-1)$;
     we know $\O_{\P^n}(-1)|_C\cong \I_{x,x'}$. When $C$ is a smooth conic, we have $C\cong \P^1$ and $\I_{x,x'}\cong \O_{\P^1}(-2)\cong \O_{\P^n}(-1)|_C$. In either case, we read the identity \eqref{H^1 NIxx' vanish} as
    \begin{equation}\label{H^1 NN O-1 vanish}
    H^1(C, \NN_{C/X}\otimes\O_C(-1))=0
    \tag{B}
    \end{equation}
    where $\O_C(-1)=\O_{\P^n}(-1)|_C$.

    When $C$ is a smooth conic, we know it is $1$-free by definition \ref{defn of M_f U_f}. Then $M_f\neq \emp$ and we are done.

    When $C$ is a reducible conic, since it is $1$-free, we can prove that the closed immersion $i_C:C\to X$ is unobstructed (i.e., $H^1(C,i_C^* T_X)=0$). Then we can smooth it into a family of smooth conics $C_t$ preserving the 1-free condition, and we have $M_f\neq \emp$. For an explicit proof of the smoothing process, see Lemma \ref{smoothing 1-free reducible conic}.
\end{proof}

Now we study how to avoid a codimension $2$ closed subvariety $Z$ in $X$ using $1$-free smooth conics.

\begin{defn}\label{defn of M(Z)}
    Let $Z\subset X$ be a codimension $2$ closed subvariety. 
    
    (i) We define the closed subscheme $M(Z)\subset M$ parameterizing conics meeting $Z$ by
\[
M(Z)= p(q^{-1}(Z))
\]
as the scheme-theoretic image where $p:\U \to M$ and $q:\U \to X$ are the projections.

(ii) Let $M_{x,x'}(Z)=M_{x,x'}\cap M(Z)\subset M_{x,x'}$ be the space parameterizing conics passing through $x,x'$ and meeting $Z$.  

(iii) Let $M_{f,x,x'}=M_f\cap M_{x,x'}$, and set $M_{f,x,x'}(Z)=M_{f,x,x'}\cap M_{x,x'}(Z)$.
Thus $M_{f,x,x'}(Z)\subset M_{f,x,x'}$ is the closed subscheme parameterizing $1$-free smooth conics passing through $x,x'$ and meeting $Z$.
\end{defn}

\begin{prop}\label{formal dim obs Mxx'fZ}
    Let $X\subset \P^n$ be a smooth projective geometrically integral variety and let $Z\subset X$ be a codimension $2$ closed subvariety. Suppose that $M_f\neq \emp$. Then there exists an open dense subvariety $V\subset X\times X$ such that for all $(x,x')\in V$, the space $M_{f,x,x'}(Z)$ satisfies
    \[\dim M_{f,x,x'}(Z)\le \dim M_f - 2\dim X+1.\]
\end{prop}
\begin{proof}
    Construct the incidence scheme
    \[
    \I_f=\{ (C,x,x',z)\in M_f\times X\times X\times Z\mid x,x'\in C, z\in Z\cap C \}
    \]
    equipped with the projections $\pi_Z: \I_f \to Z$ and $\pi_{X\times X}: \I_f\to X\times X$.

    Then each fiber is
    \[
    \pi_Z^{-1}(z)=\{(C,x,x')\in M_{f,z}\times X\times X\mid x,x'\in C\}
    \]
    where $M_{f,z}=q_f^{-1}(z)$ with $q_f:\U_f \to X$. Then the dimension of the fiber is $\dim \pi_Z^{-1}(z)=\dim M_{f,z}+2$.

    By Lemma \ref{smooth mor for U V free}, we know $q_f: \U_f \to X$ is a smooth morphism. Therefore
    \[
    \dim M_{f,z}\le \dim \U_f -\dim X =\dim M_f+1 -\dim X.
    \]
    And the projection map $\pi_Z$ yields
    \begin{equation}\label{dim of incidence I f}
    \dim \I_f \le \dim \pi_Z^{-1}(z) +\dim Z \le \dim M_f+1.
    \tag{A}
    \end{equation}
    The fiber $\pi_{X\times X}^{-1}(x,x')$ is 
    \[
    \pi_{X\times X}^{-1}(x,x')=\{ (C,z)\in M_{f,x,x'}\times Z\mid z\in C \}=q_f^{-1}(Z)\cap (M_{f,x,x'}\times X).
    \]
    Therefore $M_{f,x,x'}(Z)=p_M(\pi_{X\times X}^{-1}(x,x'))$.

    We can apply generic flatness to $\pi_{X\times X}$ to derive that
    \[
    \dim\pi_{X\times X}^{-1}(x,x') \le \dim \I_f -2\dim X
    \]
    for general points $x,x'$ (it is empty if $\pi_{X\times X}$ is not dominant). Combining the identity \eqref{dim of incidence I f}, we have
    \[
    \dim M_{f,x,x'}(Z) \le \dim\pi_{X\times X}^{-1}(x,x') \le \dim M_f -2\dim X+1.
    \]
\end{proof}

\begin{lem}\label{formally avoid obs Z}
        Let $X\subset \P^n$ be a smooth projective geometrically integral variety and let $Z\subset X$ be a codimension $2$ closed subvariety, and set $U=X\setminus Z$.
        
        (i) Suppose that $M_f\neq \emp$. Then there exists an open dense subvariety $V\subset X\times X$ such that for all $(x,x')\in V$, the space $M_{f,x,x'}(U)=M_{f,x,x'}\setminus M_{f,x,x'}(Z)$ is an open nonempty subscheme of $M_{x,x'}(U)=M_{x,x'}\setminus M_{x,x'}(Z)$.

        (ii) In particular, suppose that $\charf k=0$, and that the evaluation map
        \[
        q_2: \V \longrightarrow X\times X
        \]
        is dominant; then $M_f\neq \emp$.
\end{lem}

\begin{proof}
    Part (ii) is Theorem \ref{criterion dominant for 1-free conic}. We prove part (i).

    First, $M_{f,x,x'}(U)=M_{f,x,x'}\setminus M_{f,x,x'}(Z)$ is an open subscheme of $M_{x,x'}(U)=M_{x,x'}\setminus M_{x,x'}(Z)$ by definition. We prove that $M_{f,x,x'}(U)\neq \emp$ for general $x,x'$ provided that $M_f\neq \emp$.

    Recall that for $x\neq x'$, we have $M_{f,x,x'}=q_{2,f}^{-1}(x,x')$ with $q_{2,f}:\V^\circ_f \to X\times X$, which is a smooth morphism by Lemma \ref{smooth mor for U V free}. Also, we have $M_f\neq \emp$ and then $\V_f\neq \emp$ by assumption. 
    
    We take an irreducible component $M_0\subset M_f$ with $\dim M_0=\dim M_f$. Since $M_f$ is smooth over $k$ (see Lemma \ref{free curve are smooth}), we know $M_0\subset M_f$ is an open subscheme. Let $\V_0=p_2^{-1}(M_0)$ be the two-pointed incidence space over $M_0$, and set $\V_0^\circ=\V_0\cap\V^\circ$; then $\V_0\subset \V$ is open and $\dim \V_0=\dim M_f+2$. 
    
    The morphism $q_{2,0}: \V_0^\circ \to X\times X$, as the restriction of $q_{2,f}$ to the open subscheme $\V_0^\circ\subset \V_f^\circ$, is smooth and thus open. 
    Since $X$ is geometrically integral, we know $q_{2,0}: \V^\circ_0 \to X\times X$ is a smooth morphism with open dense image $V_1\subset X\times X$. Then, for each $(x,x')\in V_1$, we have the following dimension estimate from the smoothness of $q_{2,0}$:
    \[
    \dim M_{f,x,x'}\ge \dim \V_0-2\dim X = \dim M_f+2-2\dim X.
    \]
    Applying Proposition \ref{formal dim obs Mxx'fZ}, there exists an open dense subvariety $V_2\subset X\times X$ such that for $(x,x')\in V_2$, we have 
    \[
    \dim M_{f,x,x'}(Z)\le \dim M_f-2\dim X+1.
    \]
    Take $V=V_1\cap V_2$, which is an open dense subvariety of $X\times X$. Then for $(x,x')\in V$, we have
    \[
    \dim M_{f,x,x'}(Z) <\dim M_{f,x,x'}.
    \]
    Therefore, $M_{f,x,x'}(U)=M_{f,x,x'}\setminus M_{f,x,x'}(Z)$ is nonempty for $(x,x')\in V$. This completes the proof.
\end{proof}

\section{Conics and Lines on Complete Intersections}

In this section, we introduce Pan's theorem \cite{Pan2018}, which states that for a general pair $(x,x')\in X\times X$ on a smooth complete intersection $X\subset \P^n$ of sufficiently large dimension relative to its multidegree, the moduli space $M_{x,x'}$ parameterizing conics passing through $x,x'$ is a smooth complete intersection in $\P^{n-2}$. This enables the application of the circle method on $M_{x,x'}$.

Under this dimension hypothesis, general pairs of points on $X$ are connected by conics. Applying the criterion \ref{criterion dominant for 1-free conic} in Section 3, we know that there exist $1$-free smooth conics on $X$. We also quote some results concerning the structure of $F_x$ parameterizing lines through $x$ over a complete intersection, which will be used in Section 5 to find real $1$-free reducible conics.

\begin{conv}
    Throughout this section, let $X\subset \P^n$ be a complete intersection of type $(d_1,\dots, d_c)$ with $S=\sum_{i=1}^c d_i$ and $\O_X(1)=\O_{\P^n}(1)|_X$. We will always assume $d_i\ge2$ for $i=1,\dots,c$.
\end{conv}

The following form of Pan's theorem \cite{Pan2018} enables the application of the circle method on $M_{x,x'}$, which is critical for us.
\begin{thm}\cite{Pan2018} \label{Mxx' of com int1}
     Let $X\subset \P^n$ be a smooth complete intersection of type $(d_1, \dots, d_c)$ over a field $k$ with $S=\sum_{i=1}^c d_i$, $d_i\ge2$ for $i=1,\dots,c$. Suppose that $\charf k=0$ and $n\ge 2S-c+1$. Then there exists an open dense subvariety $W\subset X\times X$ such that for any $k$-rational point $(x,x')\in W$, the morphism
     \[
     \varphi_{x,x'}: M_{x,x'}\longrightarrow \P^{n-2}=\P^n/\operatorname{Span}(x,x')
     \]
    associating the conic $[C]\in M$ to the point 
    \[
    [\operatorname{Span}C]\in \P^{n-2}=\P^n/\operatorname{Span}(x,x')
    \]  
    is a closed immersion. Furthermore, its image is a smooth complete intersection in $\P^{n-2}$ of type $(d_{i,j})_{i,j}\setminus (1,1,2)$ with $1\le i\le c$, $1\le j\le 2d_i-1$ with degrees 
     \[
     d_{i,j} \in \{1, 1, 2, 2, \dots, d_i-1, d_i-1, d_i\}\; \text{ for each fixed $i$}.
     \]
\end{thm}

Here are some differences between our version \ref{Mxx' of com int1} and Pan's original theorem; we clarify them below.
\begin{rem}
    
    (i) Although the original theorem \cite[Theorem 1.1]{Pan2018} is stated over $\mathbb C$, the proof is
    entirely algebraic and uses no property specific to $\mathbb C$; hence the same argument applies over any algebraically closed field of characteristic zero.

    In our case, we set $K=\overline{k}$. Assuming the theorem over $K$, we show how to derive a $k$-version of the theorem by some standard descent arguments.

    By the theorem over $K$, there exists a Zariski open subscheme $W_0 \subset X_K\times X_K$ satisfying the property in the theorem. Through standard Galois descent, we have an open dense subset $W\subset X\times X$ such that $W\times_k K\subset W_0$ by part (i) of Lemma \ref{descent of open and degree1}.
     Now for each geometric point $(x,x')\in W$, we know $M_{x,x'}$ is a fiber over $W_0$ and therefore satisfies the property of the theorem.

    Now we let $(x,x')$ be a $k$-rational point of $W$, and we consider the $k$-scheme $M_{x,x'}$ and the morphism $\varphi_{x,x'}: M_{x,x'}\to \P^{n-2}$ over $k$ in the theorem. We can regard $(x,x')$ as a geometric point by base change, so $M_{x,x',K}=M_{x,x'}\times_k K$ satisfies the property of the theorem.
    This means $\varphi_{x,x',K}: M_{x,x',K}\to \P^{n-2}_K$ is a closed immersion with image being a smooth complete intersection with corresponding degrees.
    Now it comes from the descent of ideals and regular sequences in $S=k[X_0,\dots, X_{n-2}]$ (see part (ii) of Lemma \ref{descent of open and degree1}) that $\varphi_{x,x'}: M_{x,x'}\to \P^{n-2}$ is a closed immersion with image being a smooth complete intersection with corresponding degrees.
    This descends the whole theorem to the field $k$.

    (ii) In the original statement by Pan, the theorem was stated for the Kontsevich stable map stack; however, as he claims at the beginning of the article:
     "\textit{Since $p$ and $q$ are general, the line $\overline{pq}$ is not contained in $X$. The stable maps parametrized by $\mathcal{F}$ are immersions and their images in $X$ are reducible conics passing through $p$ and $q$. Therefore, we hope it would not cause any confusion if we consider $\mathcal{F}$ as the Hilbert scheme parametrizing conics contained in $X$ and passing through $p$ and $q$.}" 
     (here, $\mathcal{F}$ is the general fiber of Kontsevich moduli space over $p,q$; the word 'reducible' should be 'reduced', as an obvious typo). 

     In the article \cite{Pan2018}, Pan defines the moduli space as a Hilbert scheme in Section 4 and uses a morphism in \cite[Proposition 6.5]{Pan2018} from the fiber $\mathcal{F}$ of the Kontsevich moduli stack to the Hilbert scheme parameterizing good conics passing through $p,q$, which he called the moduli interpretation. 
    
We clarify Pan's argument that the fiber $\mathcal{F}$ can be regarded as a Hilbert scheme explicitly,
essentially by showing that all stable maps in the Kontsevich space $\overline{\M}_{0,2}(X,2)$ sending two marked points to points $x,x'$ with $x\neq x'$ and $\langle x,x'\rangle \not\subset X$ are closed immersions. See Lemma \ref{Kontse to Hilb} in Appendix A.

\end{rem}

As a corollary of Theorem \ref{Mxx' of com int1} and Theorem \ref{criterion dominant for 1-free conic}, we can show the existence of a $1$-free smooth conic on a smooth complete intersection with moderately large dimension. Also, we have shown that a $1$-free smooth conic can be deformed to avoid a codimension $2$ obstruction in Lemma \ref{formally avoid obs Z}.
\begin{thm}\label{free curve and Mfxx' avoid Z1}
     Let $X\subset \P^n$ be a smooth complete intersection of type $(d_1, \dots, d_c)$ over $k$ with $S=\sum_{i=1}^c d_i$. Let $Z\subset X$ be a closed subvariety of codimension $2$, and set $U=X\setminus Z$.

     Suppose that $\charf k=0$ and $n\ge 2S-c+1$. Then the following statements hold.

     (i) There exists a $1$-free smooth conic on $X$, i.e., $M_f\neq \emp$.

     (ii) Furthermore, there exists an open dense subvariety $V\subset X\times X$ such that for all $(x,x')\in V$, the space $M_{f,x,x'}(U)=M_{f,x,x'}\setminus M_{f,x,x'}(Z)$ is an open and nonempty subscheme of $M_{x,x'}$.
 \end{thm}

\begin{proof}
By Theorem \ref{Mxx' of com int1}, we know the evaluation map
\[
q_2:\V \longrightarrow X\times X
\]
has general fiber isomorphic to a smooth complete intersection defined by $2S-c-3$ equations in $\P^{n-2}$. Since $n\ge 2S-c+1$, such a fiber is nonempty and has dimension $n-2S+c+1\ge 2$. Thus $q_2$ is dominant. Then statements (i) and (ii) follow from Lemma \ref{formally avoid obs Z}.
\end{proof}

Recall that $q: \L \to X$ is the projection from the universal family $\L$ of the line space $F$. We consider the fiber $F_x=q^{-1}(x)$ parameterizing lines passing through a point $x\in X$.

The following description of $F_x$ can be obtained by local coordinate computation. Since a line passing through a general point is free by \cite[II.Theorem 3.11]{Kollar1996}, the corresponding $F_x$ must be a smooth complete intersection. This was observed in \cite{dejong2008familiesrationallysimplyconnected}. Also, it was formally stated in \cite[Theorem 6.2]{DeLand2015}.
\begin{prop}\label{descrip fano line F x1}
     Let $X\subset \P^n$ be a smooth complete intersection of type $(d_1, \dots, d_c)$ over $\C$ with 
     $S=\sum_{i=1}^c d_i$. Suppose that $n\ge S+1$. Then the following statements hold:

    (i) The space $F_x$ parameterizing lines on $X$ passing through a point $x$ has the following closed immersion
    \[
    t_x: F_x \longrightarrow \P(T_x\P^n),\quad [L]\mapsto [T_xL]
    \]
    with image being a scheme-theoretic zero locus of homogeneous polynomials $F_{i,j}$ in $\P^{n-1}$, where $1\le i\le c$, $1\le j\le d_i$, with degrees $\deg F_{i,j} \in \{1, 2,\dots ,d_i\}$ for each fixed $i$.
    
    Furthermore, those linear forms $F_{i,1}$ with $1\le i\le c$ exactly define the projective tangent space
    $\P(T_x X)$ of $X$ at $x$ in $\P(T_x\P^n)$.

    (ii) There exists an open dense subvariety $V\subset X$ such that all lines passing through $x\in V$ are $0$-free (i.e. $F_x=F_{f,x}$), and the image of $F_x$ under $t_x$ is a smooth complete intersection of multidegree as above.
\end{prop}

\section{Smoothing and Real Conics}

In this section, we show how to find a complex line $L$ through a real point of $X(\R)$ such that the reducible conic $C_0=L\cup \overline{L}$ is $1$-free. Then we smooth the real $1$-free reducible conic $C_0$ into a family of real $1$-free smooth conics. This provides an abundance of smooth real conics which will be the real points of $M_{x,x'}$ for suitable pairs of real points $(x,x')$. See Theorem \ref{Real $1$-free conics from conjugate free lines}.

Our results are based on \cite{Kollar1999}, which uses smoothing of a reducible curve glued from Galois conjugates of a free curve to find a rational curve defined over a local field.

\bigskip

We require the following smoothing lemma by Koll\'{a}r.
\begin{lem}\label{smoothing by Kollar}\cite[Proposition 2.2]{Kollar1999}
    Let $T$ be a Noetherian scheme with a closed point $0 \in T$ and residue field $K$. Let $h : S \to T$ be a proper and flat morphism and $Z \subset S$ a closed subscheme such that $h : Z \to T$ is also flat. Let $g : Y \to T$ be a smooth morphism. Let $f : S_0 \to Y_0$ be a $K$-morphism and $p : Z \to Y$ a $T$-morphism such that $f|_{Z_0} = p|_{Z_0}$. Assume that
\begin{enumerate}
    \item $H^1(S_0, f^*(T_{Y_0}) \otimes I_{Z_0}) = 0$.
\end{enumerate}
\textit{Then there are}
\begin{enumerate}
    \setcounter{enumi}{1}
    \item \textit{a scheme $T'$ with a $K$-point $0$,}
    \item \textit{an \'etale morphism $(0 \in T') \to (0 \in T)$, and}
    \item \textit{a $T'$-morphism $F : S \times_T T' \to Y \times_T T'$,}
\end{enumerate}
\textit{such that $F|_{S_0} = f$ and $F|_{Z \times_T T'} = p \times_T T'$.}
\end{lem}

We need the following lemma to show that the family derived by smoothing a closed immersion is generally a closed immersion.

\begin{lem}\cite{StacksProject_05XA} \label{property_closed_imm_is_open}
Consider a commutative diagram
\[
\begin{tikzcd}[column sep=large, row sep=large]
X \arrow[rr, "f"] \arrow[dr, "p"'] && Y \arrow[dl, "q"] \\
& Z &
\end{tikzcd}
\]
of schemes. Assume that
\begin{enumerate}
    \item $p$ is locally of finite type,
    \item $p$ is universally closed, and
    \item $q : Y \to Z$ is separated.
\end{enumerate}
Then there exists an open subscheme $W \subset Z$
such that a morphism $Z' \to Z$ factors through $W$ if and only if
the base change
\[
f_{Z'} : X_{Z'} \longrightarrow Y_{Z'}
\]
is a closed immersion.
\end{lem}

In Section $3$, our proof of the criterion \ref{criterion dominant for 1-free conic} for the existence of a $1$-free smooth conic requires smoothing a $1$-free reducible conic into a $1$-free smooth conic. We do this first to illustrate the smoothing process.

Recall that a $1$-free reducible conic $C$ satisfies
\[
H^1(C,\NN_{C/X}\otimes \O_C(-1))=0.
\]
We can verify that this is enough to conclude $H^1(C,i_C^* T_X)=0$ by a standard computation using the tangent-normal sequence of a nodal curve $C$; see Lemma \ref{vanish implication for red conic} in Appendix C. We show how to apply the smoothing lemma \ref{smoothing by Kollar} to obtain a family of $1$-free smooth conics.

\begin{lem}\label{smoothing 1-free reducible conic}
Let $k$ be an algebraically closed field, and let $X\subset \mathbb P^n_k$ be a smooth projective variety. Let $C=L\cup L'\subset X$
be a reducible conic. Suppose that
\[
H^1(C,\NN_{C/X}\otimes \O_C(-1))=0.
\]
Then there exists a smooth conic $C_{t'}\subset X$ satisfying
\[
H^1(C_{t'},\NN_{C_{t'}/X}\otimes \O_{C_{t'}}(-1))=0.
\]
\end{lem}

\begin{proof}

\textit{Step 1: The smoothing family.}

By assumption, we have $H^1(C,\NN_{C/X}\otimes \O_C(-1))=0$; this implies that $C\hookrightarrow X$ is unobstructed, i.e., 
\[
H^1(C,i_C^{*}T_X)=0.
\]
This follows from Lemma \ref{vanish implication for red conic}. Hence, we can apply the smoothing lemma \ref{smoothing by Kollar} by Koll\'{a}r.

The reducible conic $C=L\cup L'$ spans a plane $\mathbb P^2_k\subset\mathbb P^n_k$.
After choosing homogeneous coordinates on this plane, we may identify $C=\Proj k[u,v,w]/(uv)$.
Set $T=\mathbb A^1_k=\operatorname{Spec}k[t]$ with distinguished point $0=(t)$. We define the $T$-family
\[
S= \operatorname{Proj} k[t][u,v,w]/(uv-tw^2), \text{ with projection } h:S\to T.
\]
Its special fiber is $S_0=C$.

Take $Y=X\times T, g:Y\longrightarrow T$ to be the projection, and take $Z=\emp$. The initial morphism $f_0:C\longrightarrow X$ is the inclusion $i_C$.

By the smoothing lemma, there exist a scheme $T'$ with a $k$-point $0$, an \'{e}tale morphism $\pi:(0\in T')\longrightarrow(0\in T)$, and a $T'$-morphism
\[
F:S\times_T T' \longrightarrow X\times T'
\]
whose restriction to the special fiber is the original inclusion 
$F_0=i_C:C\hookrightarrow X$. We denote $\CC=S\times_T T'$.

Now, for any $t'\in T'(k)$ mapping to $t\in T(k)$, we have the fiber
\[
C_{t'}\cong S\times_{T}T'\times_{T'} \{t'\}\cong S\times_T \{\pi(t')\}\cong S\times_T \{t\}.
\]
By our selection of $S$ and $T$, we have
\begin{equation}\label{type of Ct}
C_{t'}\cong \Proj k[u,v,w]/(uv-tw^2),
\tag{A}
\end{equation}
which is smooth if $t\neq 0$.

\medskip
\textit{Step 2: A general $f_{t'}:C_{t'}\to X$ is a closed immersion as a 1-free smooth conic.}

Now we prove that the morphisms $f_{t'}:C_{t'}\to X$, which are the fibers of $F$ at $t'\in T'(k)$, are closed immersions for general $t'$.

Consider the morphism $F:\CC\to X\times T'$. Since $\CC\to T'$ is proper and $X\times T'\to T'$ is separated,
we can apply Lemma \ref{property_closed_imm_is_open} to obtain an open subscheme $W' \subset T'$ such that a morphism $T'' \to  T'$ factors through $W'$ if and only if the base change
\[
f_{T''} : \CC\times_{T'} T'' \longrightarrow X\times T''
\]
is a closed immersion. 

Since $f_0$ is a closed immersion, $W'$ is a nonempty open subset of $T'$ containing $0$. Replacing $T'$ by $W'$, still denoted by $T'$, we have
\[
F: S\times_T T' \longrightarrow X\times T'\text{ is a closed immersion}.
\]
Therefore, for all $t'\in T'(k)$, the morphism $f_{t'}:C_{t'}\to X$ is a closed immersion.

Now we show that for general $t'\in T'(k)$, the morphism $f_{t'}: C_{t'} \longrightarrow X$ satisfies $\degp f_{t'}^* \O_X(1)=2$. We can define the line bundle
\[
\LL= F^* p_X^* \O_X(1).
\]
Then $\LL$ is a line bundle on $\CC$ with $\LL|_{C_{t'}}=f_{t'}^* \O_X(1)$. Since $\CC\to T'$ is proper and flat, the Hilbert polynomial of $\LL|_{C_{t'}}=f_{t'}^* \O_X(1)$ is constant on each connected component of $T'$. Since $T'$ is smooth (\'{e}tale over $T\cong \AA^1$), we replace $T'$ by its irreducible component containing $0$, which is open in $T'$. 

Then the Hilbert polynomial of $\LL|_{C_{t'}}=f_{t'}^* \O_X(1)$ is constant for all $t'\in T'(k)$ and 
\begin{equation}\label{degree 2 const}
\degp f_{t'}^* \O_X(1)=\degp f_0^*\O_X(1)=2.
\tag{B}
\end{equation}
This indicates that $f_{t'}:C_{t'} \to X$ is a closed immersion with image having degree $2$ in $X$ for all $t'\in T'(k)$ (as we have shrunk $T'$). Then we can regard $C_{t'}$ as a conic in $X$. 

Now, the family $S\times_T T'\longrightarrow T'$ is flat and a local complete intersection morphism, as $h:S\to T$ is, while $X\times T'\longrightarrow T'$ is smooth. So the closed immersion
\[
S\times_T T'\hookrightarrow X\times T'
\]
is a regular immersion. Its normal sheaf $\NN=\NN_{\CC/(X\times T')}$ is therefore a vector bundle on $\CC=S\times_T T'$.
 Let $p_X:X\times T'\longrightarrow X$ be the projection and define the vector bundle
\[
\E=\NN \otimes F^{*}p_X^{*}\O_X(-1)
\]
on $\CC$.

For every point $t'\in T'(k)$, the compatibility of the normal bundle with base change (by the same argument as for the isomorphism \eqref{normal is N_C/X(x)} in the proof of Lemma \ref{tangent space of good U V}) gives $\NN|_{C_{t'}}\cong \NN_{C_{t'}/X}$. Consequently,
\begin{equation}\label{E Ct is NN -1}
\E|_{C_{t'}} \cong \NN|_{C_{t'}}\otimes F^{*}p_X^{*}\O_X(-1)|_{C_{t'}} \cong    \NN_{C_{t'}/X}\otimes \O_{C_{t'}}(-1)
\tag{C}
\end{equation}
for every $t'\in T'(k)$.

At the special point $0$, equation \eqref{E Ct is NN -1} gives $\E|_{C_0} \cong \NN_{C/X}\otimes\O_C(-1)$. The assumption of the lemma says precisely that
\[
H^1(C_0,\E|_{C_0})=0.
\]

By upper semicontinuity of cohomology, the locus
\[
U= \{t'\in T'\mid H^1(C_{t'},\E|_{C_{t'}})=0\}
\]
is a Zariski-open neighborhood of $0$. Replacing $T'$ by this open neighborhood, we may assume that
\begin{equation}\label{Ct is free}
H^1(C_{t'},\NN_{C_{t'}/X}\otimes \O_{C_{t'}}(-1))=0
\tag{D}
\end{equation}
for every $t'\in T'(k)$.

Since $T'$ is a smooth curve over the algebraically closed field $k$ containing the point $0$, there exists $t'\in T'(k)\setminus\{0\}$ mapping to $t\in T(k)\setminus \{0\}$. 
By \eqref{type of Ct}, \eqref{degree 2 const} and \eqref{Ct is free}, the morphism $f_{t'}: C_{t'}\to X$ is a closed immersion, whose image is a $1$-free smooth conic. This proves the lemma.
\end{proof}

Let $x\in X(\R)$ be a real point, fix a non-real $0$-free complex line $L\in F_{f,x}(\C)$, and take $C_0=L\cup \overline{L}$ as a geometrically reducible conic which can be descended to $\R$. 
In the proof of Lemma \ref{smoothing 1-free reducible conic}, we have shown how to smooth $C_0$ into a family of smooth conics. However, we require a real $1$-free smooth conic; this requires us to select some $L$ such that $C_0$ is $1$-free. Proposition \ref{descrip fano line F x1} gives a useful description to do this. 

Now, we show how to find a real $1$-free reducible conic $C_0$ and then smooth it into a family of real $1$-free smooth conics.

\begin{thm}\label{Real $1$-free conics from conjugate free lines}
Let $X\subset \P^n_{\R}$ be a smooth complete intersection of type $(d_1,\ldots,d_c)$, where $c\ge1, d_i\ge 2$. Set $S=\sum_{i=1}^c d_i$. Suppose that $n\ge 2S-c+1$. Then there exists a Zariski-open dense subvariety $V\subset X$, defined over $\R$, with the following properties:

(i) For every $x\in V(\R)$, there exists a non-real complex line $L\subset X_{\C}$ with $x\in L$ such that the reducible conic $C_0=L\cup\overline L$ is $1$-free:
\[
H^1(C_0,f^* T_X\otimes\O_{C_0}(-1))=0,
\]
where $f: C_0\to X$ is the closed immersion and $\O_{C_0}(-1)=\O_X(-1)|_{C_0}$.

(ii) Furthermore, $C_0$ admits a smoothing over $\R$ to real $1$-free smooth conics with nonempty real locus, and
\[
x\in\overline{q_f(\U_f(\R))}
\]
in the real analytic topology, where $\U_f$ denotes the universal family of $1$-free smooth conics in $X$.
\end{thm}

\begin{proof}
\textit{Convention: From Step 2 to Step 5 of this proof, we will work over $\C$, with all schemes, morphisms, tangent spaces, and bundles defined over $\C$ with the same notation.} 

\medskip
\noindent
\textit{Step 1: Real reducible conic $C_0$.}

First, we fix some notation for integers.
\[
m=n-c,\qquad a=n-S-1,\qquad r=S-c.
\]

Let $F$ be the Fano scheme of lines on $X$. For a point $x\in X_{\C}$, let $F_{x}$ be the scheme of lines through $x$ over $\C$.

By Proposition \ref{descrip fano line F x1}, there exists a Zariski open dense subvariety $V^\circ_{\C}\subset X_{\C}$ such that, for every $x\in V^\circ_{\C}$,
\[
F_x\subset \P(T_xX)\cong \P^{m-1}
\]
is a smooth complete intersection of multidegree $(2,\ldots,d_1;\ldots;2,\ldots,d_c)$, hence of dimension $a=n-S-1$. Furthermore, for $x\in V^\circ_\C$, all lines passing through $x$ are $0$-free, i.e., $F_x=F_{f,x}$.

Let $x\in X(\R)$, and let $L\in F_{f,x}(\C)$ be a $0$-free non-real complex line. Set
\[
C_0=L\cup\overline L.
\]
Since $L\neq\overline L$ and both pass through $x$, they meet exactly at $x$. 

Choose $p\in L\setminus\{x\}$; then $\overline p\in\overline L$. We put $D=p + \overline p$. The divisor $D$ is a hyperplane section of the plane conic $C_0$, so $\I_D\cong\O_{C_0}(-1)$.

Then the data of the closed immersion
\[
D\xrightarrow{i_D} C_0\xrightarrow{f} X_{\C}
\]
is stable under complex conjugation, so we can descend $f,D,C_0$ to schemes over $\R$. Set
\[
\E:=f^*T_X\otimes\O_{C_0}(-1).
\]
Our aim is to select $L\in F_{f,x}(\C)$ such that the first cohomology of $\E$ vanishes.

\medskip
\noindent
\textit{Step 2: 1-free condition for $C_0$.}

\textit{Convention: from now on, we work over $\C$ until the end of Step 5}.

The normalization sequence gives
\[
0\longrightarrow \O_{C_0} \longrightarrow \nu_*(\O_L \oplus \O_{\bar{L}}) 
    \xrightarrow{(s,s')\mapsto s(x)-s'(x)} k(x)\longrightarrow 0.
\]
Tensoring with $\E:=f^*T_X\otimes\O_{C_0}(-1)$, we find
\begin{equation}\label{normal seq for H1 E}
0\longrightarrow \E
\longrightarrow
\nu_*(T_X|_L(-p)\oplus
T_X|_{\overline L}(-\overline p))
\xrightarrow{(\xi,\xi')\mapsto \xi(x)-\xi'(x)}  T_xX
\longrightarrow0,
\tag{A}
\end{equation}
where we identify the fiber $T_xX\otimes \O_{C_0}(-1)|_x$ with $T_xX\otimes \mathcal{I}_D|_x\cong T_xX$, since $x\notin D$.

Now $L$ and $\overline{L}$ are $0$-free by assumption, so 
\[
H^1(C_0,\nu_*(T_X|_L(-p)\oplus
T_X|_{\overline L}(-\overline p)))=0.
\]
Define
\begin{equation}\label{defn of AL as image}
A_L:=
\operatorname{Im}(
H^0(L,T_X|_L(-p))
\longrightarrow T_xX
).
\tag{B}
\end{equation}
Reading the cohomology sequence of \eqref{normal seq for H1 E}, we have
\begin{equation}\label{AL surj to H1 vanish}
H^1(C_0,f^*T_X\otimes\O_{C_0}(-1))=0 \Longleftrightarrow A_L+A_{\overline{L}}=T_xX.
\tag{C}
\end{equation}

Now we turn to finding such an $L\in F_{f,x}(\C)$ satisfying $A_L+A_{\overline{L}}=T_xX$.

\medskip
\noindent
\textit{Step 3: $A_L$ as the positive part of the tangent bundle.}

Fix $x\in V^\circ_{\C}$ and $L\in F_{f,x}$. Since $L$ is free, the splitting type of $\NN_{L/X}$ is nonnegative. The normal sequence for
$L\subset X\subset\P^n$ is
\[
0\longrightarrow \NN_{L/X} \longrightarrow \NN_{L/\P^n}
\longrightarrow \NN_{X/\P^n}|_L \longrightarrow 0.
\]
As $X\subset \P^n$ is a smooth complete intersection, by standard computations, we have
\[
\NN_{L/\P^n}\cong\O_{\P^1}(1)^{\oplus(n-1)},
\qquad
\NN_{X/\P^n}|_L\cong\bigoplus_{i=1}^c\O_{\P^1}(d_i).
\]
Thus every summand of $\NN_{L/X}$ has degree at most $1$. We have
\[
\deg \NN_{L/X}=(n-1)-S=a,
\qquad
\operatorname{rank} \NN_{L/X}=n-c-1=a+r.
\]
Since $L$ is free, all summands are semi-positive. Therefore
\[
\NN_{L/X} \cong \O_{\P^1}(1)^{\oplus a} \oplus \O_{\P^1}^{\oplus r}.
\]

The tangent normal sequence
\[
0\longrightarrow T_L \longrightarrow T_X|_L \longrightarrow \NN_{L/X}
\longrightarrow 0
\]
splits, because
\[
\operatorname{Ext}^1(\NN_{L/X},T_L)=
H^1(L,\NN_{L/X}^{\vee}\otimes\O_{\P^1}(2))
=0.
\]
Hence
\[
T_X|_L\cong \O_{\P^1}(2) \oplus
\O_{\P^1}(1)^{\oplus a} \oplus \O_{\P^1}^{\oplus r}.
\]
We denote by
\begin{equation}\label{positive tangent}
\O_{\P^1}(1)^{\oplus a}\cong \NN_{L/X}^+ \subset \NN_{L/X},\quad
\O_{\P^1}(2) \oplus \O_{\P^1}(1)^{\oplus a}\cong (T_X|_L)^+ \subset T_X|_L
 \tag{D}
\end{equation}
the maximal ample subbundles.

Twisting the splitting identity of $T_X|_L$ above by $\O_L(-p)$, we obtain
\[
T_X|_L(-p) \cong
\O_{\P^1}(1) \oplus \O_{\P^1}^{\oplus a}
\oplus \O_{\P^1}(-1)^{\oplus r}.
\]
Now evaluate global sections at $x$. The image $A_L$ of the evaluation map (see definition \ref{defn of AL as image}) is exactly the fiber at $x$ of the first $a+1$ summands of $T_X|_L(-p)$ above. Namely,
\begin{equation}\label{AL as positive tangent}
A_L=(T_X|_L)^+|_x.
\tag{E}
\end{equation}

\medskip
\noindent
\textit{Step 4: $\P(A_L)$ as embedded tangent space of $F_x$ in $\P(T_xX)$.}

By Proposition \ref{descrip fano line F x1}, we have the following tangent morphism as a closed immersion:
\[
t_x: F_x \longrightarrow \P(T_x X),\quad [L]\mapsto [T_xL].
\]
We can regard $F_x$ as a closed subscheme of $\P(T_xX)$ via $t_x$ with each point being the tangent direction of a line passing through $x$.  By \eqref{AL as positive tangent}, we know $A_L$ is the positive part of the tangent bundle $T_X|_L$ at $x$.

Computing the differential of $t_x$ at $[L]$ as in \cite[Propositions 1.4, 2.3]{Hwang2001}, one obtains that the image of the differential $d_{[L]}t_x$ is exactly the tangent space to $\P((T_X|_L)^+|_x)$ at $[T_xL]$.  We can also verify that our $F_x$ parameterizing lines through $x$ is exactly the variety of minimal rational tangents (VMRT) at $x$ in the sense of \cite{Hwang2001}\cite{KebekusSolaConde2006} and quote this result. See Lemma \ref{VMRT tangent app} in Appendix A for this approach.

We illustrate the proof without using the language of VMRT.

\bigskip

By standard interpretation of deformation (see Lemma \ref{tangent space of good U V}), we have 
\begin{equation}\label{TL Fx}
T_{[L]} F_x\cong H^0(L,\NN_{L/X}(-x)).
\tag{F}
\end{equation}

By a similar interpretation for projective space $\P(V)$ (a special case of \cite[Theorem 3.5]{EisenbudHarris2016}), we have canonical isomorphisms
\begin{equation}\label{proj tangent formula}
T_{[L]} \P(V)\cong \Hom(L,V/L)\cong V/L\otimes L^\vee.
\tag{G}
\end{equation}
Here $V$ is any vector space and $L\subset V$ is any one-dimensional subspace. 

Consider the (first jet) evaluation map
\begin{equation}\label{ev x defn}
\ev _x: H^0(L,\NN_{L/X}(-x))\longrightarrow \NN_{L/X}(-x)|_x
\tag{H}
\end{equation}
whose image is exactly the positive part 
\[
\Im \ev_x =\NN_{L/X}^+(-x)|_x.
\]
(See also \eqref{defn of AL as image} and \eqref{AL as positive tangent}.)

Applying the identification \eqref{proj tangent formula}, we have
\begin{equation}\label{iden tang for}
    \begin{aligned}
    &\NN_{L/X}^+(-x)|_x \cong \NN_{L/X}^+|_x\otimes T_x^\vee L\cong T_{[T_xL]} \P((T_X|_L)^+|_x),\\
    &\NN_{L/X}(-x)|_x\cong \NN_{L/X}|_x\otimes T_x^\vee L \cong T_{[T_xL]}\P(T_xX).
    \end{aligned}
    \tag{I}
\end{equation}

Under the identifications \eqref{TL Fx} and \eqref{iden tang for}, the evaluation map \eqref{ev x defn} induces the map $j_x$, whose image is described in \eqref{defn of jx} below.
\begin{equation}\label{defn of jx}
j_x: T_{[L]}F_x \longrightarrow T_{[T_xL]}\P(T_xX),\quad \Im j_x=T_{[T_xL]} \P((T_X|_L)^+|_x).
\tag{J}
\end{equation}

Under the standard identifications of the tangent spaces, the differential of the tangent morphism at $[L]$
\[
d_{[L]}t_x: T_{[L]}F_x \longrightarrow T_{[T_xL]}\P(T_xX)
\]
coincides with the first-jet evaluation map $j_x$ defined above as follows from the infinitesimal deformation calculation in \cite[Proof of Proposition 1.4]{Hwang2001}.

Therefore, the identity \eqref{defn of jx} shows that the image of the differential $dt_x$ at $[L]$ is exactly
\[
 T_{[T_xL]} t_x (F_x) = T_{[T_xL]} \P((T_X|_L)^+|_x)
\]
Regarding $F_x$ as the closed subscheme of $\P(T_xX)$ via $t_x$, in terms of the embedded projective tangent space, we have
\begin{equation}\label{TLFx subset PTX}
 {\mathbb{T}}_{[L]}F_x = \P ((T_X|_L)^+|_x)\subset \P(T_xX).
 \tag{K}
\end{equation}

And we obtain the description of $A_L$ as an embedded tangent space by identities \eqref{AL as positive tangent} and \eqref{TLFx subset PTX}:
\begin{equation}\label{AL is TL Fx}
\TT_{[L]}F_x= \P(A_L),
\qquad
\widehat{\TT}_{[L]}F_x=A_L,
\tag{L}
\end{equation}
where $\widehat{\TT}_{[L]}F_x$ is the affine cone over $\TT_{[L]}F_x$.

\noindent
\medskip

\textit{Step 5: Tangent images span $T_xX$ by Terracini's Lemma.}

Via the tangent morphism $t_x$ as a closed immersion, we know
\[
F_x\subset\P(T_xX)\cong\P^{a+r}
\]
is a smooth nondegenerate complete intersection (nondegenerate means its ideal in $\P^{a+r}$ does not contain any linear form) by Step 1. So its secant variety has the expected dimension by \cite[Corollary~1.4.4]{Russo2008}:
\[
\dim\Sec(F_x)=\min\{2a+1,a+r\}.
\]
Since $a\ge r$ as $n\ge 2S-c+1$, we know
\begin{equation}\label{Sec is whole space}
\Sec(F_x)=\P(T_xX).
\tag{M}
\end{equation}

Define the locus
\[
\Omega_x:=
\{
(L_1,L_2)\in(F_{x}\times F_{x})\setminus\Delta_{F_{x}}\mid 
\widehat{\mathbb T}_{[L_1]}F_x+
\widehat{\mathbb T}_{[L_2]}F_x
= T_xX
\}.
\tag{N}
\]
The locus is Zariski open in $F_x\times F_x$, since the condition is a maximal-rank condition.

Then Terracini's lemma
\cite[Corollary~1.3.2(2)]{Russo2008} states that for general $(L_1,L_2)\in F_x\times F_x$,
\[
\dim \Sec(F_x)=2\dim F_x-\dim (\mathbb{T}_{[L_1]}F_x\cap \mathbb{T}_{[L_2]}F_x).
\]
Since $F_x$ is smooth, we have 
\begin{align*}
\dim \Sec(F_x)&=2\dim F_x-\dim (\mathbb{T}_{[L_1]}F_x\cap \mathbb{T}_{[L_2]}F_x)\\
                &=\dim \mathbb{T}_{[L_1]}F_x+\dim \mathbb{T}_{[L_2]}F_x-\dim (\mathbb{T}_{[L_1]}F_x\cap \mathbb{T}_{[L_2]}F_x)\\
                &=\dim \langle \mathbb{T}_{[L_1]}F_x ,\mathbb{T}_{[L_2]}F_x \rangle.
\end{align*}
Since $\Sec(F_x)=\P(T_xX)$ by \eqref{Sec is whole space}, we have
\begin{equation}\label{Terracini tangent span}
\langle \mathbb{T}_{[L_1]}F_x ,\mathbb{T}_{[L_2]}F_x \rangle=\P(T_xX),\quad \widehat{\mathbb T}_{[L_1]}F_x+
\widehat{\mathbb T}_{[L_2]}F_x
=T_xX.
\tag{O}
\end{equation}
This indicates that $\Omega_x$ is nonempty. Since $F_x$ is irreducible, $\Omega_x$ is Zariski open and dense in $F_x\times F_x$.

As $x\in V^\circ_\C$, lines $L_1,L_2$ through $x$ are both free by Step 1. Then, for every $(L_1,L_2)\in\Omega_x$, the description \eqref{AL is TL Fx} gives
\begin{equation}\label{L1 L2 in Omega f spans TxX}
(L_1,L_2)\in\Omega_x \Longrightarrow A_{L_1}+A_{L_2}=T_xX.
\tag{P}
\end{equation}

\medskip
\noindent

\textit{Step 6: Descent to $\R$: anti-diagonal is dense}

Now, we descend the arguments from Step 5 over $\C$ to the real field $\R$. We use the standard properties of Weil restriction as recalled in \cite[\S4.6]{poonen2017rational}. The idea of using Weil restriction to choose Galois-conjugate objects satisfying general-position conditions is classical; see, for instance, \cite[II.\S15.1]{Manin1986}.

\bigskip
Let $\sigma$ denote complex conjugation on $X_{\C}$, and replace $V^\circ_{\C}$ by $V_{\C}:=V^\circ_{\C}\cap\sigma(V^\circ_{\C})$.
Then $V_{\C}$ is conjugation-invariant, Zariski open and dense, so it descends to a Zariski-open dense subset
\[
V\subset X
\]
defined over $\R$. Fix $x\in V(\R)$, and set $F_{x,\C}=F_{x,\R}\times_\R \C$.

Consider the Weil restriction
\[
R_x:=\res _{\C/\R}(F_{x,\C}).
\]
After base change to $\C$, there is a natural isomorphism
$(R_x)_{\C}\cong F_{x,\C}\times F_{x,\C}$. And under this identification
\[
R_x(\R)=F_x(\C)
\]
maps to the anti-diagonal
\begin{equation}\label{anti-diagonal as Weil-res}
\Delta_\sigma:=
\{(L,\overline L)\in F_x(\C)\times F_x(\C)\mid L\in F_x(\C)\}.
\tag{Q}
\end{equation}
Since $F_{x,\C}$ is smooth and irreducible, $(R_x)_{\C}$ is smooth and irreducible.
Thus $R_x$ is smooth and geometrically irreducible over $\R$. 
Moreover $R_x(\R)=F_x(\C)\neq\emptyset$. Since a smooth real variety $Y$ with nonempty real points and irreducible $Y_\C$ has $Y(\R)$ Zariski dense in $Y_\C$ (see \cite[Corollary 2.2.10]{Mangolte2020}), we know $R_x(\R)$ is Zariski dense in $R_{x,\C}$.
Under the identification $(R_x)_{\C}\cong F_{x,\C}\times F_{x,\C}$, this real locus is the anti-diagonal $\Delta_\sigma$; hence
\[
\Delta_\sigma \text{ is Zariski dense in } F_{x,\C}\times F_{x,\C}.
\]

Since $\Omega_x\subset F_{x,\C}\times F_{x,\C}$ is Zariski open and dense,
\[
\Delta_\sigma\cap\Omega_x\neq\emp.
\]
We can choose $L\in F_{x}(\C)$ such that $(L,\overline L)\in\Omega_x$.
Since $\Omega_x\subset F_{x,\C}\times F_{x,\C}\setminus \Delta_{F_{x,\C}}$, the line $L$ is non-real. Now $(L,\overline L)\in \Omega_x$, so by \eqref{L1 L2 in Omega f spans TxX}, we have
\begin{equation}\label{AL ALbar spans TxX}
A_L+A_{\overline L}=T_x X_\C.
\tag{R}
\end{equation}

\medskip
\noindent
\textit{Step 7: Smoothing family of real 1-free smooth conics}

Finally, for each $x\in V(\R)$, we can choose $L\in F_{x}(\C)$ such that $(L,\overline L)\in\Omega_x$ by 
Step 6. Combining \eqref{AL ALbar spans TxX} and \eqref{AL surj to H1 vanish}, we have
\[
H^1(C_0,f^*T_X\otimes\O_{C_0}(-1))=0
\]
for $C_0=L\cup \overline{L}$ over $\C$, thus over $\R$.

Hence, $f:C_0\to X$ is unobstructed by Lemma~\ref{vanish implication for red conic}:
\begin{equation}\label{real C0 unobs}
H^1(C_0,f^*T_X)=0.
\tag{S}
\end{equation}

The remaining part is similar to the proof of Lemma \ref{smoothing 1-free reducible conic}, so we give a brief proof.

After choosing real coordinates on the plane spanned by $C_0$, write
\[
C_0\cong
\operatorname{Proj}\R[u,v,w]/(u^2+v^2).
\]
Consider the parameter scheme $T=\mathbb{A}_\R^1$ and the proper flat $T$-family 
\begin{equation}
S=
\operatorname{Proj}
\R[t][u,v,w]/(u^2+v^2-tw^2)
\tag{T}
\end{equation}
with projection $p:S\to T$. Set $Z=\emptyset$, $Y=X\times T$, and take the initial morphism $f:C_0\to X$. The initial morphism $f:C_0\to X$ is unobstructed by \eqref{real C0 unobs}.

Therefore the smoothing lemma \ref{smoothing by Kollar} applies: after an \'etale base change $T'\to\mathbb A^1_{\R}$ near $0$, the map $f:C_0\to X$ extends to a family
\begin{equation}\label{Family F}
F:S\times_{T}T'\longrightarrow X\times T'.
\tag{U}
\end{equation}
Set $\mathcal{E}=F^* p_X^* (T_X\otimes \O_X(-1))$ and $\mathcal{L}=F^* p_X^* \O_X(1)$. Then
\[
\mathcal{E}|_{C_{t'}}=f_{t'}^*(T_X\otimes \O_X(-1)), \quad \mathcal{L}|_{C_{t'}}= f_{t'}^* \O_X(1).
\]
Then the locus $W\subset T'$ consisting of points $t'\in T'$ satisfying
\begin{itemize}
    \item (1) $f_{t'}: C_{t'}\longrightarrow X $ is a closed immersion (open by Lemma \ref{property_closed_imm_is_open});
    \item (2) $H^1(C_{t'},f_{t'}^*T_X\otimes \O_{C_{t'}}(-1))=0$ (open by upper semicontinuity applied to $\mathcal{E}$);
    \item (3) $\mathrm{deg} f_{t'}^* \O_X(1)=2$ (open since the Hilbert polynomial of $\mathcal{L}|_{C_{t'}}$ is locally constant)
\end{itemize}
is a Zariski open subvariety of $T'$ containing $0$ by the reasons given above.

After shrinking $T'$ to $W$, for each $t'\in T'$ such that $C_{t'}$ is a smooth conic, the morphism $f_{t'}:C_{t'}\to X$ is a closed immersion with image of degree $2$ in $X$, and 
\[
H^1(C_{t'},f_{t'}^*T_X\otimes \O_{C_{t'}}(-1))=0.
\]

This identifies $C_{t'}$ as a $1$-free smooth conic in $X$. (Our original definition of $1$-free using the normal bundle is equivalent to this one on a smooth conic, see remark \ref{rem of free defn}.)

Since $T'\to T$ is \'{e}tale at $0$ with $T(\R)\cong \R$, the induced map on real manifolds is a local homeomorphism near $0$. For parameters corresponding to $t\in \R_{>0}$ near $0$, the conic
\[
C_{t'}\cong \Proj \R[u,v,w]/(u^2+v^2-tw^2)
\]
is a real $1$-free smooth conic containing real points by the arguments above.

Moreover, such real points can be chosen to approach the node of $C_0$ as $t\to0^+$ in the total space $S\times_T T'$. By continuity of $F$ in \eqref{Family F}, their images approach $x$. Therefore
\begin{equation}
x\in\overline{q_f(\U_f(\R))}
\tag{V}
\end{equation}
in the real analytic topology.
\end{proof}

Applying Theorem \ref{Real $1$-free conics from conjugate free lines}, we have the following theorem which is what we need for the connecting lemma \ref{connecting lemma int qua} at real places.
\begin{thm}\label{real conic through x,x'}
   Let $X\subset \P^n$ be a smooth complete intersection of type $(d_1, \dots, d_c)$ with $d_i\ge2$ over $\R$. Put $S=\sum_{i=1}^c d_i$. Suppose that $n\ge 2S-c+1$ and $X(\R)\neq \emp$. Set the open subset
   \[
   Y=q_{2,f}(\V_f^\circ(\R))\subset (X\times X)(\R).
   \]
    For $(x,x')\in Y$, we have $M_{x,x'}(\R)\neq \emp$ by definition. Then the following statements hold:

    (i) $Y$ is nonempty; furthermore, the image of the projection
   \[
    W'=p_1(Y) \subset X(\R)
   \]
   is open dense in $X(\R)$.
   
   (ii) $W'\subset X(\R)$ is an open dense subset such that for all $x\in W'$, there exists a nonempty open subset depending on $x$,
   \[
    W'_{x}\subset X(\R),
   \]
 such that for all $x'\in W'_{x}$, we have $ M_{x,x'}(\R)\neq \emp$.
\end{thm}

\begin{proof}

    Consider the projection map which is smooth by Lemma \ref{smooth mor for U V free}:
    \[
    q_{2,f}: \V_f^\circ \longrightarrow X\times X.
    \]
    We set $Y=q_{2,f}(\V_f^\circ(\R))\subset (X\times X)(\R)$ as the image, which is open in $(X\times X)(\R)$ since smooth morphisms are open in the analytic topology by the implicit function theorem.
    
    Set $W'=p_{1}(Y)$ which is open in $X(\R)$, where $p_1: X(\R)\times X(\R)\to X(\R)$ is the projection.

    As $n\ge 2S-c+1$, we can apply Theorem \ref{Real $1$-free conics from conjugate free lines} to obtain that
    \[
    q_{f}: \U_{f}(\R)\longrightarrow X(\R)
    \]
    has dense image in the analytic topology.
   
    For any $(C,x)\in \U_f(\R)$, $C$ is a smooth conic with $x\in C(\R)$; therefore $C\cong \P^1_{\R}$ and there exists $x'\in C(\R)$ with $x\neq x'$. So the projection $t_1$ sending $(C,x,x')$ to $(C,x)$
    \[
    t_{1}: \V_f^\circ(\R)\longrightarrow \U_f(\R)\text{ is surjective.} 
   \]
    Since $q_{f}\circ t_{1}=p_{1}\circ q_{2,f}$, where $p_1:X(\R)\times X(\R)\to X(\R)$ is the projection, we have $W'=p_{1}\circ q_{2,f}(\V_f^\circ(\R))= q_{f}\circ t_{1}(\V_f^\circ(\R))$ by definition. Since $t_{1}$ is surjective and $q_{f}$ has dense image, $W'\subset X(\R)$ is open dense.
    
    Now for each $x\in W'$, since $W'=p_{1}(Y)$, where $Y\subset (X\times X)(\R)$ is open, there exists an open subset $W'_{x}\subset X(\R)$ such that $\{x\}\times W'_{x}\subset Y$. Since $Y=q_{2,f}(\V_f^\circ(\R))$, for each $x'\in W'_{x}$ there exists a real point $(C,x,x')\in \V_f^\circ(\R)$. This indicates that $C\in M_{x,x'}(\R)$, and thus $M_{x,x'}(\R)\neq \emp$.
\end{proof}

\section{Arithmetic Purity}

In this section, we state our main theorem \ref{APSA of com int}. We also introduce our connecting lemma \ref{connecting lemma int qua}, which is the basic tool of our approach.

Then we introduce Wooley's theorem for the existence of $p$-adic solutions, and the theorem of Frei and Madritsch on weak approximation for smooth complete intersections by the circle method under the hypothesis that the dimension is sufficiently large relative to the degree. 

Combining the results in this article, we can prove arithmetic purity for those smooth complete intersections of large dimension relative to their degree.

Also, there are some acceptable bounds when the bound for local solutions is small; we give some examples in Corollary \ref{table of NAP}.

\bigskip
To prove arithmetic purity, we introduce the connecting lemma, which is a basic tool of our approach.
\begin{lem}\label{connecting lemma int qua}
Let $U$ be a smooth geometrically integral variety over a global field $k$ satisfying weak approximation.

Assume that there exist a finite subset $S^\dagger \subset \Omega_k$ and an open dense subset $\prod_{v\in S^\da} W_v'\subset \prod_{v\in S^\da} U(k_v)$
 such that, for any $x\in U(k)\cap \prod_{v\in S^\da} W_v'$, there exists a nonempty open subset $\prod_{v\in S^\dagger} W'_{v,x}\subset \prod_{v \in S^\dagger} U(k_v)$ depending on $x$ such that for each $x'\in U(k)\cap \prod_{v\in S^\dagger} W'_{v,x}$, there exists
 \[
  \text{a morphism } f:\P^1\to U \text{ defined over } k \text{ connecting }x,x'
  \]
 (more precisely, some $a, b\in \P^1(k)$ with $f(a)=x$ and $f(b)=x'$). Then $U$ satisfies strong approximation.
\end{lem}

\begin{proof}
Let $\U$ be an integral model of $U$ over $\O_T$ for some finite set of primes $T$ of $k$ with $T\cup \infty_k \supset S^\dagger$, and let
\[
\prod_{v\in T\cup \infty_k} W_v\times\prod_{v\notin T} \U(\O_v)
\]
be a non-empty open subset of $U({\A}_k)$. Since $U$ satisfies weak approximation, there exists $x\in U(k)$ such that $x$ lies in the following nonempty open set
\[
\prod_{v\in(T\cup \infty_k) \setminus S^\da} W_v \times \prod_{v\in S^\da} (W_v\cap W_v')
\]
where we use that $W_v'\subset U(k_v)$ is open dense.

Let $S\supset T$ be a finite set of primes of $k$ such that $x\in \U(\O_v)$ for all $v\not \in S$. By weak approximation for $U$, there exists $x'\in U(k)$ lying in
\[
\prod_{v\in S\setminus T} \U(\O_v) \times \prod_{v\in S^\dagger}W_{v,x}'
\]
which is nonempty by our assumption on $\prod_{v\notin T} \U(\O_v)$. Then there exists a morphism $f:\P^1\to U$ over $k$ with $f(a)=x$ and $f(b)=x'$, where $a,b\in \P^1(k)$ by our assumption. Since $f$ can be extended over $\O_{T'}$ for a finite set of primes $T'\supset S$ of $k$, the following set 
\[
(\prod_{v\in T\cup \infty_k} f^{-1}(W_v)) \times (\prod_{v\in S\setminus T} f^{-1}(\U(\O_v))) \times (\prod_{v\in T'\setminus S} f^{-1}(\U(\O_v))) \times \prod_{v\not \in T'} \P^{1}(\O_v)
\]
is a non-empty open subset of $\P^1({\bf A}_k)$. 
Then there exists $c\in \P^1(k)$ lying in the above open set by weak approximation for $\P^1$. Therefore, $f(c)\in U(k)$ is inside $\prod_{v\in T\cup \infty_k} W_v\times\prod_{v\notin T} \U(\O_v)$ as required.
\end{proof}

A system with a number of variables sufficiently large compared to the degree and number of polynomials admits a nontrivial solution over $p$-adic local fields. We introduce a well-known result of Wooley \cite{Wooley1998}.

Let $F$ be a system consisting of $r$ homogeneous polynomials of degree $d$ defined over a field $K$. Let $v(F)$ be the number of variables in $F$.
Let $v_{d,r}(K)$ be the largest integer among $v(F)$ for those systems $F$ defined over $K$ having no non-trivial zero. If $k$ is a number field, we set 
\[
v^{\fin}_{d,r}(k)=
\max_{v\in \Omega_{k,\fin}}
 v_{d,r}(k_v).
\]
\begin{thm}\cite{Wooley1998}
    \label{Wootheorem}
     If $K$ is a finite extension of some $\Q_p$, then
    \[
    v_{d,r}(K)\le r^{2^{d-1}}e^{2^{d+2}(\log\:d)^2}.
    \]
\end{thm}

Although the theorem of Wooley only tackles systems of forms of equal degree, we can apply the theorem to systems of forms with different degrees by the following observation based on the simple fact that $f(x)=0\Longleftrightarrow f(x)^d=0$. 
\begin{equation}\label{System replacing}
\text{System }(f_1,\dots,f_m) \text{ is solvable} \Longleftrightarrow \text{System }(f_1^{d/d_1},\dots, f_m^{d/d_m}) \text{ is solvable}
\tag{$\dagger$}
\end{equation}
where $d=\mathrm{lcm}(d_1,\dots,d_m)$. Therefore, we may set $d=\mathrm{lcm}(d_1,\dots,d_m)$ and apply the theorem of Wooley to the system $(f_1^{d/d_1},\dots, f_m^{d/d_m})$ with equal degree $d$.

The following theorem of Frei and Madritsch on forms of different degrees over all number fields, derived by the circle method, is well known.
\begin{thm}\cite[Theorem 1.3]{Frei_2016}\label{frmadritsch thm}
    Let $X\subset \P^n$ be a smooth complete intersection of type $(d_1,\dots, d_c)$ defined over a number field $k$. Suppose that 
    \[n\ge (S-1)2^{S},
    \]
    where $S=\sum_{i=1}^c d_i$. Then $X$ satisfies the Hasse principle and weak approximation.
\end{thm}

For a complete intersection $X\subset \P^n$ of type $(d_1,\dots,d_c)$, we can assume $d_i\ge 2$ by eliminating linear equations. Also, in the boundary case $c=1,d_1=2$, $X$ is a quadric hypersurface; the fiber $M_{x,x'}$ of the conic space is a projective space; we can set $v_{2,0}^{\fin}(k)=0$ in the condition of the following theorem.

\enlargethispage{2\baselineskip}
\smallskip
The following theorem is our main theorem.

\begin{thm}\label{APSA of com int}
Let $X\subset \P^n$ be a smooth complete intersection of type $(d_1,\dots,d_c)$ over a number field $k$ with $d_i\ge2$ for $1\le i\le c$. Suppose that $n\ge N_{\AP}(d_1,\dots,d_c;k)$, where $N_{\AP}(d_1,\dots,d_c;k)=\max(v^{\fin}_{\D,r}(k)+2,D2^D,2S-c+1)$,
\[
d=\max(d_1,\dots,d_c),\quad \mathcal{D}=d!,\quad r=2\sum_{i=1}^{c}d_i-c-3,\quad D=\sum_{i=1}^c d_i^2-4,\quad S=\sum_{i=1}^c d_i.
\]
Then $X$ satisfies arithmetic purity of strong approximation.
\end{thm}

\begin{proof}[Proof of Theorem \ref{APSA of com int}]
    Let $Z\subset X$ be a closed subvariety of codimension $2$ and let $U=X\setminus Z$. Let $x,x'\in X(k)$ be two distinct $k$-rational points. We have the following moduli spaces of conics defined over $k$.
    \begin{enumerate}
        \item $M=\Hilb_{2t+1}(X)$ is the Hilbert scheme of conics on $X$.
        \item $M_f\subset M$ is the open subscheme parametrizing $1$-free smooth conics.
        \item $M_{x,x'}$ is the closed subscheme of $M$ parameterizing conics passing through $x,x'$. 
        \item $M_{x,x'}(U)=M_{x,x'}\setminus M_{x,x'}(Z)$ is the space of conics in $M_{x,x'}$ that avoid $Z$ and, equivalently, lie on $U$.
        \item $M_{f,x,x'}(U)=M_f \cap M_{x,x'}(U)$ parameterizes $1$-free smooth conics passing through $x,x'$, lying on $U$.
    \end{enumerate}

    Now, we gather results from throughout this paper to prove that $U$ satisfies strong approximation by the connecting lemma \ref{connecting lemma int qua}.

    \bigskip
    \noindent\textit{(AP1) Weak approximation for $X$}. This comes from $D=\sum_{i=1}^c d_i^2-4\ge \sum_{i=1}^c d_i=S$ and the fact that the numerical hypothesis $n\ge D2^D\ge (S-1)2^S$ (in the case $c=1,d_1=2$, we use $n\ge 2S-c+1=(S-1)2^S=4$) implies that the smooth complete intersection $X$ satisfies weak approximation by Frei-Madritsch's theorem \ref{frmadritsch thm}.

    \bigskip
    \noindent\textit{(AP2) The space $M_{x,x'}$ is a smooth complete intersection}. By the assumption, we have $n\ge 2S-c+1$. Applying Pan's Theorem \ref{Mxx' of com int1}, there exists an open dense subvariety 
    \[
    V_{P}\subset X\times X
    \] such that for all $(x,x')\in V_{P}$, the space $M_{x,x'}$ is a smooth complete intersection of multidegree $(d_{i,j})_{i,j}\setminus (1,1,2)$ in $\P^{n-2}$, where $d_{i,j}\in \{1, 1, 2, 2, \dots, d_i-1, d_i-1, d_i\}$ for each fixed $i$ .

    \bigskip
    \noindent\textit{(AP3) The space $M_{x,x'}$ has p-adic solutions}. Now denote by $(A_{i,j})$ the system of polynomials defining $M_{x,x'}$ in $\P^{n-2}$ (see fact \textit{(AP2)}); these polynomials have degrees dividing $\mathcal{D}=d!$ where $d=\max(d_1,\dots,d_c)$.
     The number of variables in $A_{i,j}$ is $n-1$, and the number of polynomials is $r=2\sum_{i=1}^{c} d_i-c-3$. Replacing each polynomial $A_{i,j}$ with $A_{i,j}^{\mathcal{D}/\deg(A_{i,j})}$, 
     by the observation \eqref{System replacing}, under the hypothesis $n\ge \max(v^{\fin}_{\D,r}(k)+2,D2^D,2S-c+1)$, we have
    \[
          (x,x')\in V_P(k) \Longrightarrow \prod_{v\in \Omega_{k,\fin}}M_{x,x'}(k_v)\neq \emp.
    \]

    \bigskip
    \noindent\textit{(AP4) The space $M_{x,x'}$ has real solutions}. When $\Omega_{k,\mathrm{real}}=\emp$, we set $S^\da=\{v_0\}$ for a chosen $v_0$, and $W'_{v_0}=W'_{v_0,x}=X(k_{v_0})$ temporarily. If $\Omega_{k,\mathrm{real}}\neq \emp$, we set $S^\da=\Omega_{k,\mathrm{real}}$. Now we suppose that $\Omega_{k,\mathrm{real}}\neq \emp$. Since $n\ge 2S-c+1$, we can apply Theorem \ref{real conic through x,x'} to derive an open dense subset
   \[
   \prod_{v\in S^\da}W'_{v}\subset \prod_{v\in S^\da} X(k_v),
   \]
    such that for any $x\in X(k)\cap \prod_{v\in S^\da}W'_{v}$, there exists a nonempty open subset depending on $x$
   \[
   \prod_{v\in S^\da}W'_{v,x}\subset \prod_{v\in S^\da} X(k_v)
   \]
    such that for all $x'\in X(k)\cap \prod_{v\in S^\da}W'_{v,x}$, we have $\prod_{v\in S^\da} M_{x,x'}(k_v)\neq \emp$.

    \bigskip
    \noindent\textit{(AP5) The space $M_{x,x'}$ has Zariski dense $k$-rational points}. By facts \textit{(AP3)} and \textit{(AP4)}, under the hypothesis $n\ge \max(v^{\fin}_{\D,r}(k)+2,D2^D,2S-c+1)$, we have
    \[
 (x,x')\in V_P(k)\cap (\prod_{v\in S^\da} (W'_{v}\times W'_{v,x})) \Longrightarrow M_{x,x'}(\A_k)\neq \emp.
    \]

Here, we use that $M_{x,x'}$ is proper, so an adelic point is exactly a tuple of $k_v$ points. For all $(x,x')\in V_P(k)$, the space $M_{x,x'}$ is a smooth complete intersection of degrees $(d_{i,j})_{i,j}\setminus (1,1,2)$, where $d_{i,j}\in \{1, 1, 2, 2, \dots, d_i-1, d_i-1, d_i\}$ for each fixed $i$. Thus the circle method is applicable. Set $D=\sum_{i,j}d_{i,j}-4=\sum_{i} d_i^2-4$. Combining Theorem \ref{frmadritsch thm} and the existence of local solutions, under the hypothesis that $n\ge \max(v^{\fin}_{\D,r}(k)+2,D2^D,2S-c+1)$, we have
\[
(x,x')\in V_P(k)\cap (\prod_{v\in S^\da} (W'_{v}\times W'_{v,x})) \Longrightarrow M_{x,x'} \text{ has Zariski-dense $k$-points.}
\]

\bigskip
\noindent\textit{(AP6) Smooth conics avoid $Z$: $M_{f,x,x'}(U)$ is open nonempty}. Since $n\ge 2S-c+1$, 
applying Theorem \ref{free curve and Mfxx' avoid Z1}, there exists an open dense subvariety 
$V_f\subset X\times X$ such that for all $(x,x')\in V_f$, the space $M_{f,x,x'}(U)$ is an open nonempty subscheme of $M_{x,x'}(U)\subset M_{x,x'}$, i.e.,
\[
    (x,x')\in V_f \Longrightarrow M_{f,x,x'}(U)\subset M_{x,x'} \text{ is open nonempty.}
\]
\bigskip

Now the remaining part of the proof is clear. We take the intersection and shrink open subsets of $X$ in the facts above to
verify the condition of the connecting lemma \ref{connecting lemma int qua} for $U=X\setminus Z$.

Take $Y=V_f\cap V_P \cap (U\times U)$ (see \textit{(AP2)(AP6)}), which is open dense in $U\times U$, and set $Y_1=p_1(Y)$, where $p_1:U\times U\to U$ is the projection. So $Y_1$ is open dense in $U$ since $p_1$ is flat. Since $Y$ is open, for each $x\in Y_1$, set $i_x: U\to U\times U, \; x'\mapsto (x,x')$ and $Y_x= i_x^{-1}(Y)$ which is open nonempty in $U$, thus being open dense. By definition, we have $\{x\}\times Y_x\subset Y$ for $x\in Y_1$.

Now we shrink open subsets $W'_v$ and $W'_{v,x}$ of $X(k_v)$ with $v\in S^\da$ in \textit{(AP4)} above. Since $Y_1\subset U$ is open dense and $U$ is smooth and geometrically integral, we know that $Y_1(k_v)\subset U(k_v)$ is open dense.
Then we shrink $W'_v$ by setting it to $W'_v\cap Y_1(k_v)$ which is open dense in $U(k_v)$, keeping the same notation. Then we have $U(k)\cap \prod_{v\in S^\da} W'_v\subset U(k)\cap \prod_{v\in S^\da} Y_1(k_v)\subset Y_1(k)$.

For each $x\in U(k)\cap \prod_{v\in S^\da} W'_v\subset Y_1(k)$, we shrink $W'_{v,x}$ by setting it to $W'_{v,x}\cap Y_{x}(k_v)$, which is open nonempty since $Y_{x}\subset U$ is open dense, keeping the same notation. Therefore, if $x'\in U(k)\cap \prod_{v\in S^\da}W'_{v,x}\subset U(k)\cap \prod_{v\in S^\da} Y_{x}(k_v)\subset Y_{x}(k)$, we have $(x,x')\in \{x\}\times Y_x\subset Y$.

Finally, under the hypothesis that $n\ge \max(v^{\fin}_{\D,r}(k)+2,D2^D,2S-c+1)$, for each $x\in U(k)\cap \prod_{v\in S^\da} W'_{v}\subset Y_1(k)$, there exists an open nonempty subset $\prod_{v\in S^\da} W'_{v,x} \subset \prod_{v\in S^\da} U(k_v)$ such that, for $x'\in U(k)\cap \prod_{v\in S^\da} W'_{v,x}\subset Y_x(k)$, we have 
\[
(x,x')\in Y(k)\cap (\prod_{v\in S^\da} (W'_{v}\times W'_{v,x}))   \subset V_f(k)\cap V_P(k)\cap (\prod_{v\in S^\da} (W'_{v}\times W'_{v,x})).
\]
Therefore $M_{x,x'}$ has Zariski-dense rational $k$-points by \textit{(AP5)}, and the space $M_{f,x,x'}(U)$ is an open nonempty subscheme of $M_{x,x'}$ by \textit{(AP6)}. Hence, $M_{f,x,x'}(U)\subset M_{x,x'}$ is open nonempty, and then $M_{f,x,x'}(U)(k)\neq \emp$. This means there exists a smooth conic $C$ on $X$ defined over $k$, passing through the points $x,x'\in X(k)$, lying on $U$.
Since $C(k)\neq \emp$ (as $x\in C(k)$), we have $\P^1_k\cong C$, defining a morphism $\P^1_k\cong C\hookrightarrow U$ sending two $k$-rational points to $x,x'\in U(k)$. This verifies the conditions of the connecting lemma \ref{connecting lemma int qua}. Hence, strong approximation for $U$ follows.
\end{proof}

By substituting the inequality from Theorem \ref{Wootheorem} into the condition of our main theorem \ref{APSA of com int}, we obtain the following simplified coarse bound in the form of a tower of three exponentials:
\begin{thm}\label{APSA uneffective}
Let $X\subset \P^n$ be a smooth complete intersection over a number field $k$.
If
\[
 \dim X \ge 2^{2^{2^{(\degp X)^2}}},
\]
then $X$ satisfies arithmetic purity of strong approximation.
\end{thm}
We mention that any smooth projective geometrically integral variety $X$ over a number field $k$ satisfying $\dim X\ge (\degp X) 2^{\degp X}$ is a smooth complete intersection (see \cite{BrowningHeathBrown2017}), so our theorem above is valid for these varieties under the large dimension hypothesis.

\bigskip
The uniform bound given by Theorem \ref{Wootheorem} is very large. However, for some low degrees, we may derive some acceptable bounds in our main Theorem. We quote some bounds which are lower than those in more general cases. 

Let $F$ be a system consisting of $r_d+\dots+r_1$ homogeneous polynomials, of which $r_i$ have degree $i$ for $1\le i \le d$, all defined over a field $K$. Let $v(F)$ be the number of variables in $F$. We adopt the convention of \cite[\S2,Lemma 2.1]{Wooley1998} that $V_{d}(r_d,\dots,r_1)(K)$ is the largest integer among $v(F)$ for those systems $F$ defined over $K$ having no non-trivial zero; $\phi=\phi_d(K)$ is the largest integer admitting a diagonal form of degree $d$ in $\phi$ variables over $K$ having no non-trivial zero.

\begin{lem}\label{solutions sharp bound}
    For a $p$-adic local field $K$, we have the following estimates:

(i)\cite{Martin1997} $V_2(q,0)(K)\le 2q^2+2$.

(ii)\cite{Dem50}\cite{Lew52} $V_3(1,0,0)(K)=9$.

(iii)\cite[Lemma 2.2]{Wooley1998} $V_{3}(r,q,0)(K)\le r\phi+ V_{2}(q+r^2\phi,rq\phi+r^3\phi^2)(K)$.

(iv) $V_3(r,q,0)(K)\le 162r^4+81r^3+36r^2q+9rq+2q^2+9r+2$.

(v) $V_3(1,1,0)(K)\le 23, V_3(1,2,0)(K)\le 51$.
\end{lem}

Part (iv) follows from substituting $V_2(q,b)(K)\le 2q^2+2+b$ from part (i) and $\phi_3\le V_3(1,0,0)(K)=9$ from part (ii) into the inequality in part (iii). As any quadratic form in $s\ge5$ variables is isotropic by part (i), Witt's decomposition yields a totally isotropic vector subspace of dimension at least $(s-4)/2$; hence the inequalities in (v) follow from part (ii).

\begin{lem}\label{mixed bound}\cite[Corollary 1.2]{Frei_2016}
Let $Y\subset\P^m_k$ be a smooth geometrically integral complete intersection of $q$ quadrics and $r$ cubics over a number field $k$. Set
\[
\Phi(q,r)=\max\bigl\{(q+1)(2q+8r)+8r^2,4q+16qr+8r(r+1), 8q+12r\bigr\}.
\]
Assume that
\[
\begin{aligned}
m &\ge q+r-1+\Phi(q,r) && \text{if } r>0,\\
m &\ge 2q^2+3q-1 && \text{if } r=0.
\end{aligned}
\]
Then $Y$ satisfies the Hasse principle and weak approximation.
\end{lem}

The numerical hypothesis is precisely the condition in \cite[Corollary 1.2]{Frei_2016} for a system of $q$ quadrics and $r$ cubics, using the estimate $B_{\max}\le T-1$ from the discussion following that corollary. This yields the asymptotic formula
\[
 N(P)=\mathfrak S\mathfrak J \cdot P^{\kappa(s-2q-3r)}
       +O\bigl(P^{\kappa(s-2q-3r)-\delta}\bigr)
\]
for some $\delta>0$, where $N(P)$ counts the common zeros in the homogeneously expanding boxes and ideal lattices, $\kappa=[k:\Q]$ and $s=m+1$.

Suppose that $Y(\mathbf A_k)\neq \emptyset$. Since $Y$ is a smooth complete intersection, its nonzero local affine lifts are nonsingular. The standard local-density argument therefore gives $\mathfrak S>0$ and $\mathfrak J>0$; see the discussion immediately following \cite[Theorem 1.3]{Frei_2016}. The Hasse principle and weak approximation then follow from the asymptotic formula, as explained in the same subsection.

\begin{cor}\label{table of NAP}
For the following multidegrees, the conclusion of Theorem~\ref{APSA of com int} holds with the explicit bounds $N_{\AP}(d_1,\dots,d_c)$ given below, uniformly over all number fields $k$:

    (i) $N_{\AP}(2,\dots,2)=2c^2+3c-1$.

    (ii) $N_{\AP}(3,\dots,3)=162c^4+153c^3-10c^2-6c+4$.

    (iii) $N_{\AP}(2,2)=8$.

    (iv) $N_{\AP}(3)=39$ (recovering the result of an earlier version).

    (v) $N_{\AP}(2,3)=62$.
\end{cor}

\begin{proof}
These bounds follow by revisiting the estimates in the proof of the main Theorem \ref{APSA of com int} using sharper bounds in Lemmas \ref{solutions sharp bound} and \ref{mixed bound}.

    (i) For the $(2,\dots,2)$ type, the condition that $X$ satisfies weak approximation gives the bound $2c^2+3c-1$ by Lemma \ref{mixed bound}.
    
    The space $M_{x,x'}$ is a smooth complete intersection of $2c-2$ hyperplanes and $c-1$ quadrics in $\P^{n-2}$, which is then a smooth complete intersection of $c-1$ quadrics in $\P^{n-2c}$.
     The bound for $p$-adic local solvability is $n-2c\ge 2(c-1)^2+2$ by part (i) of Lemma \ref{solutions sharp bound}; the bound for the circle method is $n\ge 2c^2+c-2$ by Lemma \ref{mixed bound}.

    (ii) For the $(3,\dots,3)$ type, the space $M_{x,x'}$ is a smooth complete intersection of $2c-1$ quadrics and $c$ cubics in $\P^{n-2c}$. 
    The bound from the circle method is $n-2c\ge (2c-1)+c-1+ \Phi(2c-1,c)$, which gives $n\ge 40c^2+5c-6$. The bound for $p$-adic solvability can be taken to be $n-2c+1> 162c^4+153c^3-10c^2-8c+4\ge V_3(c,2c-1,0)$ by Lemma \ref{solutions sharp bound}.

    (iii) Let $X$ be a smooth complete intersection of two quadric hypersurfaces. To prove the arithmetic purity of strong approximation off any $S$, we can assume $X(k)\neq \emp$. By \cite[Theorem 3.11]{CTSSD87I}, if $n\ge 6$ and $X(k)\neq \emp$, then $X$ satisfies weak approximation. 
    Now the bound for $p$-adic local solvability for $M_{x,x'}$ is $n\ge 8$; the bound from the circle method for $M_{x,x'}$ is also $n\ge 8$ by the argument in the proof of part (i).

    (iv) For the $(3)$ type, the space $M_{x,x'}$ is a smooth complete intersection of type $(2,3)$ in $\P^{n-2}$. The bound from the circle method is $n-2\ge 37$ by Lemma \ref{mixed bound}. 
    The bound for $p$-adic local solvability is $25$ by part (v) of Lemma \ref{solutions sharp bound}.

    (v) For the $(2,3)$ type, the space $M_{x,x'}$ is a smooth complete intersection of type $(2,2,3)$ in $\P^{n-4}$. The bound from the circle method is $n-4\ge 58$ by Lemma \ref{mixed bound}. 
    The bound for $p$-adic local solvability is $n\ge 55$ by part (v) of Lemma \ref{solutions sharp bound}.
\end{proof}

\section{Application: Generic Finiteness of R-equivalence}

In this section, we prove that for a smooth complete intersection $X$ of large dimension compared to its degree, the classes of $R$-equivalence on $X$ are generically finite in number.

This is a natural application of our results developed to prove the arithmetic purity: we essentially prove that for a general $k$-point $x\in X(k)$, there exists an open analytic subset of $X$ depending on $x$ in which all $k$-points can be connected by a $k$-rational smooth conic passing through $x$. Recall that two rational points on $X$ are said to be $R$-equivalent if they can be connected by a $\P^1_k$-chain, so our results can be translated to control classes of $R$-equivalence.

\bigskip
We have the following theorem.
\begin{thm}\label{R finiteness}[Theorem \ref{R generic finite}]
Let $X\subset \mathbb P^n_k$ be a smooth complete intersection of type $(d_1,\ldots,d_c)$ over a number field $k$. There exists an effective constant $N_R(d_1,\ldots,d_c)$ such that, if
\[
n\geq N_R(d_1,\ldots,d_c),
\]
then there exists a Zariski open dense subvariety $X^\circ\subset X$ such that
\[
\operatorname{Im} [X^\circ(k)\longrightarrow X(k)/R] \text{ is finite.}
\]
\end{thm}

\subsection{Notations and Previous Results}

Now we introduce some notation and results from the previous sections. We will use these results to prove Theorem \ref{R finiteness}.
Recall that we have summarized these results \textit{(AP1)-(AP6)} in the proof of our main theorem \ref{APSA of com int} for arithmetic purity. Propositions \textit{(R1)-(R6)} stated here are almost word-by-word repetitions of
 \textit{(AP1)-(AP6)} except for \textit{(R4)}. So we only state the consequences without explaining them again.

Let $X\subset \P^n$ be a smooth complete intersection of type $(d_1,\dots,d_c)$ over a number field $k$ with $d_i\ge2$ for $1\le i\le c$. We always assume the following large dimension hypothesis throughout this section:
\begin{equation}\label{defn of NR}
n\ge N_R(d_1,\dots,d_c;k),
\tag{A}
\end{equation}
where $N_R(d_1,\dots,d_c;k)=\max(v^{\fin}_{\D,r}(k)+2,D2^D,2S-c+1)$,
\[
d=\max(d_1,\dots,d_c),\quad \mathcal{D}=d!,\quad r=2\sum_{i=1}^{c}d_i-c-3,\quad D=\sum_{i=1}^c d_i^2-4,\quad S=\sum_{i=1}^c d_i.
\]

By Wooley's uniform estimate (Theorem \ref{Wootheorem}), the quantity $v^{\fin}_{\D,r}(k)$ admits an effective upper bound independent of the number field $k$. Consequently, there exists an effective constant $N_R(d_1,\dots,d_c)$ depending only on the multidegree such that $N_R(d_1,\dots,d_c;k)\le N_R(d_1,\dots,d_c)$ for every number field $k$. Thus, proving the desired conclusion under the hypothesis \eqref{defn of NR} establishes Theorem~\ref{R finiteness}.

We have the following moduli spaces of conics defined over $k$ for two distinct $k$-rational points $x,x'$ of $X$.
    \begin{enumerate}
        \item $M=\Hilb_{2t+1}(X)$ is the Hilbert scheme of conics on $X$.
        \item $M_f\subset M$ is the open subscheme parametrizing $1$-free smooth conics.
        \item $M_{x,x'}$ is the closed subscheme of $M$ parameterizing conics passing through $x,x'$. 
        \item $M_{f,x,x'}=M_f \cap M_{x,x'}$ parameterizes $1$-free smooth conics passing through $x,x'$.
        \item $\V_f^\circ$ is the incidence space over $M_f$ with two distinct marked points.
    \end{enumerate}

    \bigskip
    \noindent\textit{(R1) Weak approximation for $X$}. The smooth complete
    intersection $X$ satisfies weak approximation.

    \bigskip
    \noindent\textit{(R2) The space $M_{a,b}$ is a smooth complete intersection}. There exists an open dense subvariety 
    \[
    G_{P}\subset X\times X
    \] such that for all $(a,b)\in G_{P}$, the space $M_{a,b}$ is a smooth complete intersection of degrees $(d_{i,j})_{i,j}\setminus (1,1,2)$, where $d_{i,j}\in \{1, 1, 2, 2, \dots, d_i-1, d_i-1, d_i\}$ for each fixed $i$.

    \bigskip
    \noindent\textit{(R3) The space $M_{a,b}$ has p-adic solutions}. We have
    \begin{equation*}\label{local soluable}
          (a,b)\in G_P(k) \Longrightarrow \prod_{v\in \Omega_{k,\fin}}M_{a,b}(k_v)\neq \emp.
    \end{equation*}

\bigskip
\noindent\textit{(R4) The space $M_{a,b}$ has real solutions}. 
For every real place $v$ of $k$, identify $k_v$ with $\R$. By Theorem \ref{real conic through x,x'}, we have
\[
Y_v:=q_{2,f}\;(\V_f^\circ(k_v))
   \subset X(k_v)\times X(k_v),
\qquad
W_v:=p_1(Y_v)\subset X(k_v).
\]
The set $Y_v$ is open in the real analytic topology and $W_v$ is open and dense. If $(a,b)\in Y_v$, then there exists a $1$-free smooth conic over $k_v$ through $a$ and $b$; in particular,
\[
M_{a,b}(k_v)\neq\emp.
\]

\bigskip
  \noindent\textit{(R5) The space $M_{a,b}$ has Zariski dense $k$-rational points}. By facts \textit{(R2)(R3)(R4)} and the theorem of Frei-Madritsch, we have
  \[
 (a,b)\in G_P(k)\cap \prod_{v\in \Omega_{k,\mathrm{real}}} Y_v \Longrightarrow M_{a,b} \text{ has Zariski dense $k$-points.}
 \]
(when $\realplaces=\emp$, the intersection should be read as $G_P(k)$).

\bigskip
\noindent\textit{(R6) Existence of $1$-free smooth conics: $M_{f,a,b}$ is open nonempty}. Since $n\ge 2S-c+1$, by applying Theorem \ref{free curve and Mfxx' avoid Z1}, there exists an open dense subvariety 
$G_f\subset X\times X$ such that for all $(a,b)\in G_f$, the space $M_{f,a,b}$ is an open nonempty subscheme of $M_{a,b}$, i.e.,
\[
    (a,b)\in G_f \Longrightarrow M_{f,a,b}\subset M_{a,b} \text{ is open nonempty.}
\]
(The original statement of Theorem \ref{free curve and Mfxx' avoid Z1} states that $M_{f,x,x'}(U)\subset M_{f,x,x'}$ is nonempty for an open subvariety $U=X\setminus Z$; we do not introduce $U$ and $Z$ here and only use that $M_{f,x,x'}$ is nonempty.)

\subsection{Fixing the Zariski open subset of X}

First, we define the Zariski open dense subvariety $X^\circ \subset X$ in the assertion of Theorem \ref{R finiteness}. 

The sets $W_v\subset X(k_v)$ from \textit{(R4)} are open dense in the real topology. However, they are not Zariski open subsets of $X$. We first replace them by algebraic open subsets using some basic facts of semialgebraic sets in real algebraic geometry; see Appendix B.

For every real place $v\in\realplaces$, put
\[
k^v:=v(k)\subset\R.
\]
Along the isomorphism $v:k\xrightarrow{\sim}k^v$, we can regard the quasi-projective variety $\V_f^\circ$, the variety $X$, and the morphisms
\[
q_{2,f}:\V_f^\circ\longrightarrow X\times X,
\qquad
p_1:X\times X\longrightarrow X
\]
as varieties and morphisms defined over $k^v$. 
On a $k^v$-affine chart, the real locus of a quasi-projective $k^v$-variety is semialgebraic and defined over $k^v$, as closed conditions are polynomial equations with coefficients in $k^v$, while a Zariski open condition $g\neq 0$ is the union of the two sign conditions $g>0$ and $-g>0$.

Applying the projection lemma \ref{projection lemma} of semialgebraic sets first to $q_{2,f}$ and then to $p_1$, we conclude that
\[
Y_v=q_{2,f}\;(\V_f^\circ(\R)),
\qquad
W_v=p_1(Y_v)
\]
are semialgebraic and defined over $k^v$.

By Fact \textit{(R4)}, $W_v$ is open and dense in $X(\R)$. Applying Lemma \ref{semi alg shrink} over $k^v$ gives a nonempty Zariski open dense subset $U'_v$ of the $k^v$-model of $X$, defined over $k^v$, such that
\[
U'_v(\R)\subset W_v.
\]
Transporting $U'_v$ back along $v:k\xrightarrow{\sim}k^v$, we obtain a Zariski open dense subset $U_v\subset X$, defined over $k$, satisfying
\[
U_v(k_v)\subset W_v.
\]

Let $\iota\colon X\times X\longrightarrow X\times X,\iota(a,b)=(b,a)$ be the involution.
Define
\[
G:= (G_f\cap G_{\mathrm{P}})\cap \iota(G_f\cap G_{\mathrm{P}}).
\]
Then $G$ is a symmetric Zariski open dense subset of $X\times X$.
 For $(a,b)\in G$, the two-point conic fiber $M_{a,b}$ is the smooth projective complete intersection of \textit{(R2)}, and its locus parametrizing $1$-free smooth conics is geometrically nonempty because $G\subset G_f$.

Since the projection $p_1\colon X\times X\to X$ is an open morphism and $G$ is symmetric, set
\[
U_0:=p_1(G)=p_2(G),
\]
which is a Zariski open dense subvariety of $X$.

We now fix the Zariski open dense subvariety
\begin{equation}\label{defn of Xcirc}
X^\circ:= U_0\cap\bigcap_{v\in\realplaces}U_v.
\tag{B}
\end{equation}
of $X$. 
In the remaining part of this section, we shall show that $X^\circ$ satisfies the property in Theorem \ref{R finiteness}.

\subsection{Finiteness of real connectivity types}

Fix a real place $v$.  Interchanging the two marked points shows that $Y_v$ is invariant under the involution exchanging coordinates of $X(k_v)\times X(k_v)$.  In particular,
\[
p_1(Y_v)=p_2(Y_v)=W_v,
\]
and we can use $Y_v$ to define a relation on $W_v$. The relation is generated by the pairs $(a,b)\in W_v^2$ such that $(a,b)\in Y_v$. More precisely, we define a relation $\sim_v$ on $W_v$ by
\[
a\sim_v b
\quad\Longleftrightarrow\quad
\begin{gathered}
\exists \;m\ge 0 \text{ and }a=a_0,a_1,\dots,a_m=b\\
\text{s.t. }  (a_{i-1},a_i)\in Y_v\quad(1\le i\le m).
\end{gathered}
\]
The chains of length $0$ are allowed, which gives the reflexivity $a\sim_v a$. The symmetry of $Y_v$ gives symmetry of $\sim_v$. The definition by chains gives transitivity of $\sim_v$.
Hence $\sim_v$ is an equivalence relation on $W_v$. 

Geometrically, we know $a\sim_v b$ implies that there exists a chain of real $1$-free smooth conics connecting $a$ and $b$ over $k_v$.

The following lemma states that $\sim_v$-equivalence classes are controlled by $\pi_0(W_v)$.

\begin{lem}\label{clopen-classes}
Every $\sim_v$-equivalence class is both open and closed in $W_v$.
Consequently, every connected component of $W_v$ is contained in a single $\sim_v$-equivalence class.
\end{lem}

\begin{proof}
Let $a\in W_v$. Since $W_v=p_1(Y_v)$, there exists $b\in X(k_v)$ such that $(a,b)\in Y_v$. Since $Y_v$ is open in $X(k_v)\times X(k_v)$ (see \textit{R4}), there are real-analytic neighborhoods $A\ni a$ and $B\ni b$ satisfying
\[
A\times B\subset Y_v.
\]
For every $a'\in A$, the pairs $(a',b)$ and $(a,b)$ lie in $Y_v$.  By symmetry also $(b,a)\in Y_v$, so $a',\;b,\;a$ is a $Y_v$-chain. So $a'\sim_v a$, and the equivalence class of $a$ is open in $W_v$. The complement of an equivalence class is the union of the other equivalence classes and is therefore open as well.
Consequently, every equivalence class is open and closed. A connected component cannot meet two disjoint nonempty open and closed equivalence classes.
\end{proof}

It is known that every semi-algebraic set has only finitely many connected components.
\begin{lem}\label{finiteness of comp}
The set $W_v$ has only finitely many connected components.
\end{lem}

\begin{proof}
  We know $W_v$ is a semialgebraic subset of $X(k_v)$, thus having finitely many connected components by lemma \ref{comp gen are fin}.
\end{proof}

Now we define the map $\tau$ for a point $x\in X^\circ(k)$ which records the location of $x$ at real places up to connected components. This will control the $R$-equivalence information for $x$ at real places.

By definition \eqref{defn of Xcirc}, every $x\in X^\circ(k)$ has $x_v\in W_v$ at every real place. We therefore define its real type
\begin{equation}\label{tau real class}
\tau(x):=
([x_v])_{v\in\realplaces}
\in
\prod_{v\in\realplaces}\pi_0(W_v),
\tag{C}
\end{equation}
where $[x_v]$ indicates the connected component of $W_v$ containing $x_v$.

If $k$ has no real places, the product on the right is regarded as a singleton.

\subsection{Globalizing two points of the same real type}

Let $x,x'\in X^\circ(k)$
and assume they have the same real type
\[
\tau(x)=\tau(x').
\]
We shall prove that they are $R$-equivalent, i.e., $x\Rsim x'$.

By Lemma \ref{clopen-classes}, for every real place $v$, the points $x_v$ and $x'_v$ lie in the same $\sim_v$-equivalence class. Then there exists a finite sequence
\[
x_v=a_{v,0},a_{v,1},\dots,a_{v,m_v}=x'_v
\]
such that
\[
(a_{v,j-1},a_{v,j})\in Y_v
\qquad
(1\le j\le m_v).
\]
Every edge is realized by a $1$-free smooth real conic.

If necessary, an edge may be subdivided. Indeed, if a smooth real conic $C$ contains $a$ and $b$, then $C(\R)\neq\emp$, so $C\simeq\P^1_\R$. Choose a third real point $c\in C(\R)$ distinct from $a$ and $b$. Then both $(a,c)$ and $(c,b)$ lie in $Y_v$.
If the chain has length $0$, then $x_v=x'_v$; choose any $b$ with $(x_v,b)\in Y_v$; by symmetry, $x_v,\ b,\ x_v$ is a two-edge $Y_v$-chain.

Therefore the finitely many chains may be refined so that they all have one common length $m\ge 2$.

For the fixed endpoints $x$ and $x'$, define a Zariski open subset $\mathcal B^{(m)}_{x,x'}\subset X^{m-1}$ by
\begin{equation}\label{defn of Bxx'}
\mathcal B^{(m)}_{x,x'}:=
\{
(z_1,\dots,z_{m-1}) :
(z_{j-1},z_j)\in G\ \text{ for }1\le j\le m
\},
\tag{D}
\end{equation}
where $z_0=x$ and $z_m=x'$.

\begin{lem}\label{B open den}
The subset $\mathcal B^{(m)}_{x,x'}$ is Zariski open and dense in $X^{m-1}$.
\end{lem}

\begin{proof}
Because $x,x'\in X^\circ\subset U_0$, the slices
\[
G_x:=\{z\in X:(x,z)\in G\},
\qquad
G^{x'}:=\{z\in X:(z,x')\in G\}
\]
are nonempty Zariski open subsets of $X$, hence dense. The condition \eqref{defn of Bxx'} for $\mathcal B^{(m)}_{x,x'}$ is the intersection of
\[
G_x\times X^{m-2},
\qquad
X^{m-2}\times G^{x'},
\]
and the inverse images of $G\subset X\times X$ under the adjacent-coordinate projections
\[
p_{j-1,j}: X^{m-1}\longrightarrow X\times X \;\text{ for } \;2\le j\le m-1.
\]
Each of these is Zariski open and dense in $X^{m-1}$, so their finite intersection is Zariski open and dense.
\end{proof}

For a real place $v$, set
\[
\mathcal A_v
:=
\{
(z_1,\dots,z_{m-1})\in X(k_v)^{m-1} :
(z_{j-1},z_j)\in Y_v\ \text{for  }1\le j\le m
\},
\]
where $z_0=x_v, z_m=x'_v$.

The real conic chain chosen above shows that $\mathcal A_v\neq\emp$, and the openness of $Y_v$ implies that $\mathcal A_v$ is open in the real-analytic topology by the same argument in the proof of Lemma \ref{B open den}.

Since $X^{m-1}$ is smooth and geometrically integral, $\mathcal B^{(m)}_{x,x'}(k_v)$ is open dense in $X^{m-1}(k_v)$. Therefore
\begin{equation}\label{defn of Nv}
\mathcal N_v=\mathcal A_v\cap \mathcal B^{(m)}_{x,x'}(k_v)
\tag{E}
\end{equation}
is open nonempty in $X^{m-1}(k_v)$ for every real place $v$.

By \textit{(R1)}, $X$ satisfies weak approximation. Hence the finite product $X^{m-1}$ satisfies weak approximation as well.

Assume first that $k$ has at least one real place. Simultaneously approximate the nonempty open sets $\mathcal N_v$ from \eqref{defn of Nv} for all real places. We obtain
\[
(z_1,\dots,z_{m-1})\in X(k)^{m-1}\cap \prod_{v\in \realplaces } \mathcal N_v.
\]
Thus
\begin{align}
(z_{j-1},z_j)&\in G(k) && (1\le j\le m),\tag{F}
\label{global good}\\
(z_{j-1,v},z_{j,v})&\in Y_v && (v\in\realplaces,\ 1\le j\le m).\tag{G}
\label{real good}
\end{align}
For completeness, the first assertion follows because $X(k)^{m-1}\cap \mathcal N_v$ is contained in
 $X(k)^{m-1}\cap \mathcal B^{(m)}_{x,x'}(k_v)\subset \mathcal B^{(m)}_{x,x'}(k)$.

If $k$ has no real places, fix any place $v_0$. Because $\mathcal B^{(m)}_{x,x'}$ is Zariski open and dense in the smooth variety $X^{m-1}$, the subset $\mathcal B^{(m)}_{x,x'}(k_{v_0})\subset X^{m-1}(k_{v_0})$ is open dense, thus being nonempty.
Applying weak approximation at $v_0$, we obtain global intermediate points satisfying \eqref{global good}. No condition \eqref{real good} is then required.

Now we fix such a sequence $(z_0,\dots, z_m)\in X^{m+1}(k)$ with $z_0=x, z_m=x'$.

\subsection{Each global edge is represented by a k-rational conic}

Fix $j\in\{1,\dots,m\}$ and write
\[
a=z_{j-1}, \qquad b=z_j.
\]
By \eqref{global good} and the inclusion $G\subset G_{\mathrm{P}}$, the space $M_{a,b}$ parameterizing conics through $a,b$ is a smooth projective complete intersection by \textit{(R2)}.

At every real place $v$, \eqref{real good} gives $(a_v,b_v)\in Y_v$ and there exists a real $1$-free smooth conic through $a_v$ and $b_v$. This gives a real solution of $M_{a,b}$. The $p$-adic solutions are derived for large $n$ (see \textit{R3}), and then $M_{a,b}(\A_k)\neq \emptyset$. By Frei-Madritsch's theorem, we know
\[
M_{a,b} \text{ has Zariski-dense $k$-points}
\]
(see \textit{R5}).

Since $(a,b)\in G\subset G_f$, the open locus $M_{f,a,b}$ parametrizing $1$-free smooth conics in $M_{a,b}$ is open nonempty. Zariski density therefore supplies a $k$-point of this open locus. Hence there exists a smooth conic
\[
C_j\subset X
\]
defined over $k$ and containing $a$ and $b$.

Because $a\in C_j(k)$, the smooth conic $C_j/k$ has a $k$-point and is therefore isomorphic to $\P^1_k$. Hence $a$ and $b$ are directly $R$-equivalent. Repeating this for every edge gives
\[
x=z_0\Rsim z_1\Rsim\cdots\Rsim z_m=x'.
\]

We have proved the following.

\begin{prop}\label{real type to R-equi}
If $x,x'\in X^\circ(k)$ have the same real type $\tau(x)=\tau(x')$, then
\[
x\Rsim x' \; \text{ in } X(k).
\]
\end{prop}

\subsection{Generic finiteness of R-equivalence}

\begin{thm}
\label{R generic finite}
Under the large-dimension hypotheses of \eqref{defn of NR}, the fixed Zariski open dense subset $X^\circ\subset X$ defined in \eqref{defn of Xcirc} satisfies
\begin{equation}\label{main bound}
\#\,\Im\;(X^\circ(k)\longrightarrow X(k)/R)
\le
\prod_{v\in\realplaces}\#\pi_0(W_v)
<\infty.
\tag{H}
\end{equation}
\end{thm}

\begin{proof}
By Lemma \ref{finiteness of comp}, each $\pi_0(W_v)$ is finite. The map
\[
\tau\colon X^\circ(k)\longrightarrow \prod_{v\in\realplaces}\pi_0(W_v) 
\]
was defined in \eqref{tau real class}. By Proposition \ref{real type to R-equi}, all points in a single fiber of $\tau$ belong to one global $R$-equivalence class.
Therefore the number of $R$-equivalence classes represented by $X^\circ(k)$ is at most $\#\prod_{v\in\realplaces}\pi_0(W_v) $, which is exactly \eqref{main bound}.
\end{proof}

\begin{cor}
If $k$ is totally imaginary, then
\[
\#\,\Im\;(X^\circ(k)\longrightarrow X(k)/R)\le 1.
\]
Equivalently, whenever $X^\circ(k)\neq\emp$, all points of $X^\circ(k)$ lie in a single $R$-equivalence class on $X$.
\end{cor}

\appendix

\section{Comparison of Moduli Spaces}

In this Appendix, we compare some different moduli spaces parameterizing morphisms or curves.

We review the definition of the Kontsevich stable map stack. See \cite[Definition 3.1]{Behrend_1999}.
\begin{defn}
Fix a smooth projective variety $X\subset \P^n_k$ over a field $k$ with $\charf k=0$ and $\O_X(1)=\O_{\P^n}(1)|_X$.

The Kontsevich stable map stack of degree $2$ with two marked points is an algebraic stack
\[
\overline{\M}_{0,2}(X,2):
(\Sch/k)^{\mathrm{op}}\longrightarrow \operatorname{Grpd}.
\]
For a $k$-scheme $T$, an object (i.e., a stable map) of
$\overline{\M}_{0,2}(X,2)(T)$ consists of data $(\CC,f,x_1,x_2)$ as follows:

(1) a proper flat finitely presented family $\pi:\mathcal C\to T$ whose geometric fibers are connected at-worst-nodal curves of arithmetic genus $0$;

(2) two (ordered) sections $x_1,x_2: T\to \mathcal C$ such that for each geometric point $t\in T$, the points $x_1(t),x_2(t)$ are distinct and lie in the smooth locus of $\mathcal C_t$;

(3) a morphism $f : \mathcal C \to X$ such that for each geometric point $t\in T$, the morphism $f_t: \mathcal C_t\to X$ has degree $2$, namely 
\[ 
\deg f_t^*\mathcal{O}_X(1)=2, 
\] 

(4)(Stability) for every geometric point $ t\in T$, every (rational) irreducible component of $\mathcal C_{t}$ contracted by $f_{t}$ (i.e., mapping to a point) contains at least three special points,
where the special points are the nodes and the marked points in $\CC_t$. A morphism
\[
\varphi:(\mathcal C,f,x_1,x_2)
   \longrightarrow
(\mathcal C',f',x_1',x_2')
\]
over $T$ is a $T$-isomorphism
$\varphi:\mathcal C\to\mathcal C'$ satisfying
\[
\varphi\circ x_i=x_i'\quad (i=1,2),
\qquad
f'\circ\varphi=f.
\]
For a morphism $T'\to T$, pullback is given by base change.

There is a natural evaluation morphism
\[
\ev_{2}:\overline{\M}_{0,2}(X,2)\longrightarrow X\times X.
\]
It sends a stable map $(C,f,x_1,x_2)$ to $(f(x_1),f(x_2))$.
\end{defn}

It is known that for a connected nodal curve of genus $0$, if $C=\cup_i C_i$ with $C_i$ irreducible components, then the dual graph $\Gamma_C$ of $C$ (regarding components as vertices and nodes as edges) is a tree, and each $C_i$ is isomorphic to $\P^1$.
See \cite[Exercise8.5]{ConradAlterations} or \cite[Corollary 1.10]{Payne2020ModuliTropical}.

\begin{lem}\label{Kontse to Hilb}
Let $X\subset \P^n$ be a smooth projective variety over a field $k$ with $\charf k=0$, and set $\O_X(1)=\O_{\P^n}(1)|_X$. 
Let $x,x'\in X(k)$ be two distinct rational points such that
$\langle x,x'\rangle\not\subset X$. Let
\[
\overline{\M}_{x,x'}:=
\overline{\mathcal{M}}_{0,2}(X,2)
\times_{X\times X}\{(x,x')\}
\]
be the fiber of the Kontsevich stable map stack over $(x,x')$ under the evaluation map.

Let $M=\Hilb_{2t+1}(X)$, let $\mathcal{U}\subset M\times X$ be the universal family, and set $\mathcal{V}:=\mathcal{U}\times_M\mathcal{U}$. The two projections to $X$ induce the evaluation morphism $q_2:\mathcal{V}\longrightarrow X\times X$. Let
\[
M_{x,x'}:=q_2^{-1}(x,x').
\]

Then there is a natural isomorphism
\[
H:\overline{\M}_{x,x'}\xrightarrow{\sim}M_{x,x'}.
\]
Under this isomorphism, a stable map is sent to its image in $X$.

In particular, every stable map parametrized by $\overline{\M}_{x,x'}$ is a closed immersion, and the fiber
$\overline{\M}_{x,x'}$ is represented by the Hilbert incidence scheme $M_{x,x'}$.
\end{lem}

\begin{proof}
\textit{Step 1: Every geometric stable map in $\overline{\M}_{x,x'}$ is a closed immersion}

We first consider geometric points. Let
\[
(C,f,x_1,x_2)\in\overline{\M}_{x,x'}
\]
be a geometric stable map. Thus $C$ is a connected at-worst-nodal curve of arithmetic genus $0$, the marked points $x_1,x_2$ are distinct smooth points of $C$, and
\[
f(x_1)=x,\qquad f(x_2)=x'.
\]
Write
\[
C=\bigcup_i C_i
\]
for its irreducible components. Since $C$ is a connected at-worst-nodal curve of arithmetic genus $0$, its dual graph $\Gamma_C$ is a tree and every $C_i$ is isomorphic to $\P^1$ by our arguments under the definition of the stable map stack.

For every irreducible component $C_i$, set
\[
\beta_i:=\deg\left((f|_{C_i})^{*}\mathcal{O}_X(1)\right).
\]
Since the stable map has degree $2$, we have
\begin{equation}\label{degree 2 eq}
\sum_i\beta_i=2.
\tag{A}
\end{equation}
Moreover, if $\beta_i=0$, then $f$ is constant on $C_i$, and stability requires $C_i$ to contain at least three special points, namely nodes or marked points by stability.

We first observe that no leaf component of $\Gamma_C$ can be contracted by $f$. Such a component contains exactly one node. Hence stability would force it to contain both marked points $x_1,x_2$.
 Since $f$ is constant on this component, this would give $f(x_1)=f(x_2)$, contradicting $x\neq x'$.

There are now only two possible distributions of positive degrees in equation \eqref{degree 2 eq}.

Suppose first that there is exactly one nonconstant component $C_0$, with
\[
\beta_0=2.
\]
If $f(C_0)$ were a line $L$, then $f|_{C_0}$ would have degree $2$ onto $L$. Since every other component is contracted, the whole image $f(C)$ would be contained in $L$. In particular, $x,x'\in L$ and therefore $L=\langle x,x'\rangle\subset X$, contrary to our assumption. 

So $f(C_0)$ is a degree-$2$ integral curve. Over an algebraically closed field this is a smooth conic, and $f|_{C_0}:C_0\to f(C_0)$ has degree $1$, so it is an isomorphism.

If $C$ had any component other than $C_0$, the tree $\Gamma_C$ would have a leaf distinct from $C_0$. Such a leaf would be contracted, contradicting the observation above. We know
\begin{equation}\label{2-0 smooth conic}
C=C_0,\quad  f:C\to X \text{ is a closed immersion}
\tag{B}
\end{equation}
with image a smooth conic.

Suppose next that there are exactly two nonconstant components $C_1,C_2$, with
\[
\beta_1=\beta_2=1.
\]
Then
\[
f|_{C_1}:C_1\xrightarrow{\sim}L,
\qquad
f|_{C_2}:C_2\xrightarrow{\sim}L'
\]
are isomorphisms onto lines $L,L'\subset X$. The two lines must be distinct. Indeed, if $L=L'$, then both $x$ and $x'$ belong to $L$, and hence $L=\langle x,x'\rangle\subset X$, again a contradiction.

Since no contracted component can be a leaf of $\Gamma_C$, the only leaves are $C_1$ and $C_2$. Therefore $\Gamma_C$ is a path joining $C_1$ and $C_2$. 
If there were a contracted component between them, every such component would have precisely two nodes. Stability would then force it to carry a marked point.

All contracted components in this path are mapped to the same point $s\in L\cap L'$.

Consequently, if a contracted component exists, one of the marked points, say $x_1$, lies on such a component, and $x=f(x_1)=s$. Since $x\neq x'$, the other marked point cannot lie on a contracted component. It lies on $C_1$ or $C_2$, so one of the two lines $L,L'$ contains both $x$ and $x'$. It follows that $\langle x,x'\rangle\subset X$,
again contradicting our hypothesis. Hence there are no contracted components and
\begin{equation}\label{1-1 reducible conic}
C=C_1\cup C_2,\quad f:C_1\cup C_2\xrightarrow{\sim}L\cup L'\hookrightarrow X\text{ is a closed immersion}
\tag{C}
\end{equation}
with image a reducible conic.

We have therefore proved that for every geometric point of $\overline{\M}_{x,x'}$, the corresponding stable map is a closed immersion whose image is a reduced conic passing through $x,x'$ by \eqref{2-0 smooth conic} and \eqref{1-1 reducible conic}.
 In particular, the double-cover case does not occur.

\noindent
\medskip

\textit{Step 2: Relative isomorphism over $T$}

We now upgrade this geometric-point statement to a statement about families. Let $T$ be a scheme over $k$ and let
\[
(\pi:\mathcal{C}\to T,\,
f:\mathcal{C}\to X_T,\,
x_1,x_2)
\]
be a $T$-point of $\overline{\M}_{x,x'}$, i.e., a $T$-valued stable map, where $X_T:=X\times T$. Then
\[
f\circ x_1=x_T,\qquad
f\circ x_2=x'_T,
\]
where $x_T$ and $x'_T$ denote the two constant sections of $X_T$.

For every geometric point $ t\to T$, the preceding argument shows that
\[
f_{ t}:\mathcal{C}_{ t}\longrightarrow X_{ t}
\]
is a closed immersion. We claim that $f$ itself is a closed immersion.

According to the lemma \cite{StacksProject_05XA}, since $\mathcal{C}\to T$ is proper and locally of finite type and $X\times T\to T$ is separated, there exists an open subscheme $T'\subset T$ such that any morphism $T''\to T$ factors through $T'$ if and only if the base change 
\[
f_{T''}: \mathcal{C}\times_T T''\to X_{T''}
\]
is a closed immersion. By the arguments above, we know each geometric point $ t\to T$ factors through $T'$. Therefore $T'=T$, and then $f:\CC\to X_T$ itself is a closed immersion. 

We may therefore regard $\mathcal{C}$ as a closed subscheme of $X_T$. Since $\mathcal{C}\to T$ is flat and every geometric fiber has degree $2$ and arithmetic genus $0$, its Hilbert polynomial is $2t+1$. 
Thus it determines a $T$-point of $\Hilb_{2t+1}(X)$.
The two marked sections show that this family contains the constant sections $x_T$ and $x'_T$. Hence it determines a $T$-point of $M_{x,x'}$.

This construction is compatible with arbitrary base change in $T$ and therefore defines a natural morphism
\[
H:\overline{\M}_{x,x'}\longrightarrow M_{x,x'}.
\]

Conversely, let $T$ be a scheme and let a $T$-point of $M_{x,x'}$ be given. Equivalently, we have a flat family
\[
\mathcal{Z}\subset X_T
\]
with Hilbert polynomial $2t+1$, together with the two constant sections
\[
x_T,x'_T:T\longrightarrow\mathcal{Z}.
\]
We show that the inclusion
\[
\iota:\mathcal{Z}\hookrightarrow X_T
\]
together with these two sections defines a family of stable maps.

Every geometric fiber $\mathcal{Z}_{ t}$ is a conic in $\P^n_{ t}$, and hence is either a smooth conic, the union of two distinct intersecting lines, or a double line. The double-line case is impossible: if $\mathcal{Z}_{ t}$ were supported on a line $L$, then both $x$ and $x'$ would belong to $L$, and since $L\subset X_{ t}$ we would obtain $L=\langle x,x'\rangle\subset X_{ t}$, contrary to our assumption.

Thus every geometric fiber is reduced. If
\[
\mathcal{Z}_{ t}=L\cup L'
\]
is reducible, then $x$ and $x'$ cannot lie on the same component, since that component would be $\langle x,x'\rangle$. This also indicates that neither $x$ nor $x'$ can be the node $L\cap L'$.

It follows that every geometric fiber of $\mathcal{Z}\to T$ is a connected at-worst-nodal curve of arithmetic genus $0$, and the two sections $x_T,x'_T$ lie in its relative smooth locus and are disjoint.
Therefore
\[
(\mathcal{Z}\to T,\,
\iota:\mathcal{Z}\hookrightarrow X_T,\,
x_T,x'_T)
\]
is a family of genus-$0$ degree-$2$ stable maps. Stability is automatically satisfied, since no component is contracted. We therefore obtain a natural morphism
\[
G:M_{x,x'}\longrightarrow\overline{\M}_{x,x'}.
\]

It remains to check that the two constructions are inverse.
Starting with a flat family $\mathcal{Z}\subset X_T$ representing a $T$-point of $M_{x,x'}$, the morphism $G$ regards its inclusion into $X_T$ as a stable map, and $H$ then takes its scheme-theoretic image. This indicates
\[
H\circ G=\id_{M_{x,x'}}.
\]
Conversely, starting with a family $(\mathcal{C},f,x_1,x_2)$ in $\overline{\M}_{x,x'}(T)$, we proved above that $f$ is a closed immersion. Hence $f$ identifies $\mathcal{C}$ canonically with its scheme-theoretic image in $X_T$, and under this identification the marked sections agree with $x_T,x'_T$.
Therefore
\[
G\circ H\simeq\id_{\overline{\M}_{x,x'}}.
\]
These constructions are compatible with arbitrary base change in $T$ and are mutually inverse up to the canonical isomorphisms described above. Moreover, every object of $\overline{\M}_{x,x'}$ has trivial automorphism group, since its map to $X$ is a closed immersion. 
Hence $H$ identifies the fibered category $\overline{\M}_{x,x'}$ with the functor represented by $M_{x,x'}$. Therefore
\[
\overline{\M}_{x,x'}\cong M_{x,x'}.
\]
This proves the lemma.
\end{proof}

\bigskip
  Let $X\subset \P^n$ be a smooth complete intersection of type $(d_1, \dots, d_c)$ over $\C$ with $S=\sum_{i=1}^c d_i$ and $n\ge S+2$. Let $F$ be the Fano scheme of lines on $X$, and let $F_x$ be the closed subscheme of $F$ parameterizing lines passing through $x$. Consider the tangent morphism
   \[
    t_x: F_x \longrightarrow \P(T_xX),\quad [L]\mapsto [T_xL].
    \]
According to Proposition \ref{descrip fano line F x1}, $t_x$ is a closed immersion, with image being a smooth complete intersection of positive dimension for general $x$. 

We identify $t_x(F_x)$ with the variety of minimal rational tangents (VMRT) at $x$ in the sense of \cite{Hwang2001}\cite{KebekusSolaConde2006}, and then identify the image of $d_{[L]}t_x$ with the tangent space to $\P((T_X|_L)^+|_x)$ at $[T_xL]$, where $(T_X|_L)^+$ is the positive part of $T_X|_L$.

\begin{lem}\label{VMRT tangent app}
Let $X\subset \P^n$ be a smooth complete intersection of type $(d_1, \dots, d_c)$ over $\C$ with 
$S=\sum_{i=1}^c d_i$ and $n\ge S+2$. For general $x\in X$, we regard $F_x$ as a closed subscheme of $\P(T_xX)$ via $t_x$. Then the embedded tangent space of $F_x$ at $[L]$ satisfies
\[
 \mathbb{T}_{[L]} F_x = \P((T_{X}|_L)^+|_x).
\]
\end{lem}

\begin{proof}
We prove this from the standard VMRT (variety of minimal rational tangents) formalism, following \cite{KebekusSolaConde2006}; see also \cite{Hwang2001}.

We fix the projectivization convention. Throughout this proof, we set $\P(E):=\P_{\mathrm{sub}}(E)$ as the projective space of one-dimensional subspaces of $E$. In the article \cite{KebekusSolaConde2006}, the authors use the quotient convention, so their $\P_{\mathrm{quot}}(E^\vee)$ is canonically identified with our
$\P(E)$.

We fix some notation of integers.
\[
m=n-c,\qquad a=n-S-1,\qquad r=S-c.
\]

As $n\ge S+1$, the variety $X$ is covered by lines by Proposition \ref{descrip fano line F x1}. Therefore, the degree-one rational-curve locus contains an irreducible component whose universal family dominates $X$. Let
\[
H\subset \RatCurves^{n}(X)
\]
be the corresponding maximal family. Every member of $H$ is a projective line, so has degree $1$ (with respect to $\O_X(1)$). Since every nonconstant projective curve has positive degree, $H$ is a dominating family of rational curves of minimal degree in the sense of \cite[Definition~2.4]{KebekusSolaConde2006}.

We use the notation $[\ell]$ to denote points in $\RatCurves^n(X)$, and $\ell$ for the associated curves in $X$.
We assume $x\in X$ is a general point for the statements of \cite[Theorem~2.6, Proposition~2.8, Theorem~2.16, and Proposition~2.19]{KebekusSolaConde2006}.

Fix such an $x$. Let $H_x$ be the family of lines in $H$ passing through $x$, i.e.,
\[
H_x=\{ [\ell] \in H \mid x\in \ell \}.
\]
Let
\[
\theta_x: H_x\dashrightarrow \P(T_xX),\quad [\ell]\mapsto [T_x\ell]
\]
be the rational tangent map. In our case, every member of $H_x$ is a line, so $\ell$ is smooth at $x$ and  $\theta_x$ is a regular morphism. Let
\[
\nu:\widetilde H_x\longrightarrow H_x
\]
be its normalization, and define the tangent morphism
\[
\tau_x: \widetilde H_x \longrightarrow \P(T_xX)
\]
as the composition $\tau_x=\theta_x \circ \nu$. Then $\tau_x$ is a finite morphism by \cite[Theorem 2.16]{KebekusSolaConde2006}.

Now, the image of the tangent morphism $\tau_x$
\begin{equation}\label{defn of Cx VMRT}
\CC_x:=\tau_x(\widetilde H_x)\subset \P(T_xX)
\tag{A}
\end{equation}
is the variety of minimal rational tangents (VMRT) of $H$ at $x$ in the sense of \cite[Definitions~2.15,~2.17 and Theorem 2.16]{KebekusSolaConde2006}.

By definition of the tangent morphism, the closed points of $\CC_x$ are precisely the tangent directions at $x$ of the lines represented by $H_x$:
\begin{equation}\label{descrip of Cx}
\CC_x(\C)=
\{
[T_x\ell]\in\P(T_xX)
\mid [\ell]\in H_x(\C)
\}.
\tag{B}
\end{equation}
On the other hand, recall that $F_x$ parameterizes all lines in $X$ passing through $x$, thus parameterizing their tangent directions as a closed subscheme of $\P(T_xX)$:
\begin{equation}\label{descrip of Fx in tangent}
F_x(\C)=
\{
[T_xL]\in\P(T_xX) \mid
L\subset X \text{ is a line through } x
\}.
\tag{C}
\end{equation}
Since every curve represented by $H$ is a line,
\eqref{descrip of Cx} and \eqref{descrip of Fx in tangent} immediately give
\begin{equation}\label{Cx subset Fx}
\CC_x(\C)\subset F_x(\C).
\tag{D}
\end{equation}

We now prove this inclusion is an equality. Since $H$ is dominating, $H_x\neq \emp$ for general $x$. Choose a line
$\ell$ represented by $[\ell]\in H_x$. By
\cite[Proposition~2.8]{KebekusSolaConde2006}, $\ell$ is free. We have the splitting identity
\[
T_X|_{\ell} \cong
\O_{\P^1}(2) \oplus
\O_{\P^1}(1)^{\oplus a} \oplus
\O_{\P^1}^{\oplus r},
\]
with maximal ample subbundle
\[
(T_X|_{\ell})^+ =
\O_{\P^1}(2) \oplus \O_{\P^1}(1)^{\oplus a}.
\]

Since all curves parametrized by $H$ are lines in the fixed embedding
$X\subset\P^n$, \cite[Proposition~2.19]{KebekusSolaConde2006} applies. The proposition claims that
$\tau_x$ is an embedding, $\CC_x$ is smooth, and that the embedded projective tangent space to $\CC_x$ at the point $[T_x\ell]$ is
\begin{equation}\label{tagent space of VMRT}
\TT_{[T_x\ell]}\CC_x=
\P((T_X|_\ell)^+|_x).
\tag{E}
\end{equation}
(The original text omits the notation $|_x$. Compare the end of \cite[II.\S 6.4]{KebekusSolaConde2006}, where
$\P((T_X|_\ell)^+|_x)$ is written out explicitly.)

Therefore
\[
\dim_{[T_x\ell]}\CC_x=a.
\]

As $\CC_x$ and $F_x$ are both smooth closed subvarieties of $\P(T_xX)$, they are reduced. 
Hence the inclusion of closed-point sets \eqref{Cx subset Fx} implies $\CC_x\subset F_x$. For general $x$, $F_x$ is a smooth complete intersection of dimension $a=n-S-1>0$ by Proposition \ref{descrip fano line F x1}, thus being geometrically irreducible. Now, the tangent space description \eqref{tagent space of VMRT} shows that $\CC_x$ is smooth of dimension $a$. Therefore, we have
\[
\CC_x=F_x
\]
as closed subvarieties of $\P(T_xX)$. Since $\tau_x$ is surjective onto $\CC_x=F_x$, every $L\in F_x$ is represented by a member of $H_x$.
Consequently, for any line $L$ through $x$, there exists $[L]\in H_x$ representing $L$, and the tangent space description \eqref{tagent space of VMRT} of the VMRT $\CC_x$ can be transferred to $F_x$:
\[
    \TT_{[L]}F_x=\P((T_X|_L)^+|_x).
\]
\end{proof}

\section{Semialgebraic Sets}

In this appendix, we introduce some basic concepts and facts about semialgebraic sets over a real closed field.
 The main reference is \cite{BPR06}.

Let $R$ be a real closed field and let $K\subset R$ be an ordered subfield.  For subsets of affine space we use the usual notion of a semialgebraic set, as in \cite[\S2.3]{BPR06}.
More precisely, a basic semialgebraic subset of $R^N$ defined over $K$ means a set of the form
\[
\{
x\in R^N:
P_1(x)=\cdots=P_r(x)=0,\quad
Q_1(x)>0,\dots,Q_s(x)>0
\},
\]
with $P_i,Q_j\in K[x_1,\dots,x_N]$, and a semialgebraic subset defined over $K$ is a finite union of such basic sets.  This is equivalent to the convention in \cite[\S2.3]{BPR06}, where the common zero condition is written
as a single equation $P=0$: over a real closed field one may replace $P_1=\cdots=P_r=0$ by $P_1^2+\cdots+P_r^2=0$. We keep the common-zero notation.

For an affine $K$-variety $Y$, we use the corresponding intrinsic definition with $P_i,Q_j\in K[Y]$.  
For a $K$-variety $Y$, a subset $S\subset Y(R)$ will be called semialgebraic and defined over $K$ if this holds on a finite affine $K$-open cover of $Y$.  
This is the coefficient-sensitive analogue of the standard definition of semialgebraic subsets of the real points of a variety, which does not depend on the choice of an affine open covering. See
\cite[\S6, Definition 2]{DK81}.

We shall also use the standard functoriality of semialgebraic sets under semialgebraic maps.  
By \cite[Proposition 2.83]{BPR06}, if $f:S\to T$ is a semialgebraic map, then the image of every semialgebraic subset of $S$ is semialgebraic. For our purposes we need the slightly more precise statement that coefficients can be kept in $K$. We repeat its proof using a coefficient-preserving projection theorem
\cite[Theorem 2.76]{BPR06}.

\begin{lem}\label{projection lemma}
  Let $f:Y\to Z$ be a $K$-morphism of $K$-varieties $Y$ and $Z$, and let $S\subset Y(R)$ be a semialgebraic subset defined over $K$. Then the image $f(S)$ is a semialgebraic subset of $Z(R)$ defined over $K$.
\end{lem}

\begin{proof}
The coefficient-preserving projection theorem \cite[Theorem 2.76]{BPR06} states that for an ordered subring
$D\subset R$ (we use it with $D=K$), if
\[
A\subset R^{m+1}
\]
is semialgebraic and defined over $D$, then its projection obtained by forgetting the last coordinate is again semialgebraic and defined over $D$.  Iterating gives the same conclusion for every coordinate projection
$R^{m+n}\to R^m$.

Suppose first that $Y$ and $Z$ are affine $K$-varieties, $f:Y\to Z$ is a $K$-morphism, $S\subset Y(R)$ is semialgebraic and defined over $K$.

After choosing affine embeddings over $K$, the graph
\[
\Gamma_f\subset Y\times Z
\]
is a closed $K$-subvariety.  Hence the intersection
\[
\Gamma_f(R)\cap(S\times Z(R))
\]
of semialgebraic sets defined over $K$ is still semialgebraic and defined over $K$ by definition.
We know its projection to $Z(R)$ is exactly $f(S)$. Repeated application of \cite[Theorem 2.76]{BPR06} to the same projection shows in addition that $f(S)$ is defined over $K$.

The same statement holds for morphisms of $K$-varieties.  Cover the target by finitely many affine $K$-opens $Z_j$, and cover each $f^{-1}(Z_j)$ by finitely many affine $K$-opens $Y_{ij}$.  Then
\[
f(S)\cap Z_j(R)=
\bigcup_i f(S\cap Y_{ij}(R)),
\]
and the affine case applies to every term on the right.
\end{proof}

\begin{lem}\label{semi alg shrink}
Let $K\subset\R$ be a subfield, let $T$ be a geometrically integral variety over $K$, and let $W\subset T(\R)$ be a nonempty semialgebraic subset defined over $K$ that is open and dense in the real topology.
  Then there exists a nonempty Zariski open subset $U\subset T$, defined over $K$, such that
\[
U(\R)\subset W.
\]
In particular, $U$ is Zariski dense in $T$.
\end{lem}

\begin{proof}
Choose $x\in W$, and choose a $K$-defined affine open subset $T_0\subset T$ with $x\in T_0(\R)$.  Set
\[
B:=T_0(\R)\setminus W.
\]
Then $B$ is $K$-semialgebraic and has empty interior in $T_0(\R)$.

Write $B$ as a finite union
\[
B=\bigcup_{\alpha=1}^r S_\alpha
\]
of basic $K$-semialgebraic sets
\[
S_\alpha=
\{
t\in T_0(\R):
P_{\alpha,1}(t)=\cdots=P_{\alpha,r_\alpha}(t)=0,\quad
Q_{\alpha,1}(t)>0,\dots,Q_{\alpha,s_\alpha}(t)>0
\},
\]
with $P_{\alpha,i},Q_{\alpha,j}\in K[T_0]$.  Discard the empty pieces. For each remaining $S_\alpha$, at least one of the equalities is given by a nonzero element of $K[T_0]$.  If $r_\alpha=0$, or if all $P_{\alpha,i}$ were zero in $K[T_0]$, then the effective defining conditions for $S_\alpha$ would consist only of strict inequalities, so $S_\alpha$ would be a nonempty open subset of $T_0(\R)$, contradicting the fact that $B$ has empty interior.

For every nonempty $S_\alpha$, choose one such nonzero element $P_\alpha\in K[T_0]$ vanishing on $S_\alpha$.  If $B=\varnothing$, take $U=T_0$.  Otherwise put
\[
P:=\prod_{\alpha} P_\alpha.
\]
Since $T_0$ is integral, $K[T_0]$ is a domain, so $P\neq0$.  Moreover
\[
B\subset V(P)(\R).
\]
Thus
\[
U:=D_{T_0}(P)
\]
is a nonempty $K$-defined Zariski open subset of $T$ and satisfies $U(\R)\subset W$.  Since $T$ is integral, every nonempty Zariski open subset of $T$ is dense.
\end{proof}

\begin{lem}\label{comp gen are fin}
Let $K\subset\R$ be a subfield, let $T$ be a variety over $K$, and let $W\subset T(\R)$ be a $K$-semialgebraic subset. Then $W$ has only finitely many connected components in the real analytic topology.
\end{lem}

\begin{proof}
Choose a finite affine $K$-open cover
\[
T=\bigcup_{i=1}^r U_i.
\]
By the definition above,
\[
W_{i}:=W\cap U_i(\R)
\]
is semialgebraic and defined over $K$.  After choosing an affine $K$-embedding $U_i\hookrightarrow \A^{N_i}_{K}$, the set $W_{i}$ is a semialgebraic subset of $\R^{N_i}$.  Hence
\cite[Theorem 2.4.5]{BCR98} implies that $W_{i}$ has only finitely many connected components.

Since the cover is finite, there are only finitely many connected components of all the sets $W_{i}$. Each such component is a connected subset of $W$, hence is contained in a unique connected component of $W$. 
Conversely, every connected component of $W$ contains a point of some $W_{i}$, and therefore contains the connected component of $W_{i}$ through that point.  
Thus the finitely many connected components of the $W_{i}$ map surjectively onto $\pi_0(W)$.  In particular, $W$ has only finitely many connected components.
\end{proof}

\section{Auxiliary Algebraic Computations}

In this appendix, we state two lemmas and prove them by standard computations.

\bigskip
First, we have the following lemma by standard descent of schemes and ideals.
\begin{lem}\label{descent of open and degree1}
Let $k$ be a perfect field, and let $K=\overline{k}$ be the algebraic closure of $k$.

(i) Let $Y$ be a geometrically integral variety over $k$, and let $W_0\subset Y_K$ be an open and nonempty subvariety. Then there exists an open and nonempty subvariety $W\subset Y$ such that $W\times_k K\subset W_0$.

(ii) Let $Z$ be a $k$-scheme and $i:Z\to \P^N$ be a morphism over $k$. If $i_K:Z_K\to \P^N_K$ is a closed immersion with image being a smooth complete intersection of degrees $(d_1,\dots,d_r)$, then $i:Z\to \P^N$ is a closed immersion with image being a smooth complete intersection of degrees $(d_1,\dots,d_r)$.
\end{lem}
\begin{proof}
    (i) Since $Y$ is of finite type over $k$, we know the open subvariety $W_0\subset Y_K$ can be defined over a finite Galois extension $L|k$ (as $k$ is a perfect field), say $W_{0,L}$ with $W_{0,L}\times_L K= W_0$. Set $G=\Gal (L|k)$ and define $W_{L}=\cap_{\sigma\in G} \sigma (W_{0,L})$, which is open in $Y_L$ since $G$ is finite. Now, since $Y$ is geometrically integral, we know each $\sigma (W_{0,L})$ is open dense in $Y_L$, and thus $W_{L}\subset Y_L$ is open and nonempty. Since $W_{L}$ is stable under Galois conjugation by $\Gal(L|k)$, it can be descended to a $k$-open subscheme $W\subset Y$ with $W\times_k L=W_{L}$. Now $W\times_k K\subset W_{0,L}\times_L K= W_0$ by definition.

    (ii) We know that smoothness and the property of being a closed immersion satisfy $fpqc$ descent by \cite{StacksProject_02VL} and \cite{StacksProject_02L6}. Thus $i: Z\to \P^{N}$ is a closed immersion with smooth image. Now we show its ideal is generated by polynomials of the corresponding degrees.
    
    Let $S=k[x_0,\dots,x_N]$, and let the irrelevant maximal ideal $\m = (x_0,\dots, x_N )$, and let $I\subset S$ be the saturated ideal of $Z$ in $\P^N$. Now $i_K:Z_K\to \P^N_K$ is a closed immersion with image being a smooth complete intersection of degrees $(d_1,\dots,d_r)$ by assumption; we set $I_K=(F_1,\dots, F_r)$ with homogeneous polynomials $F_1,\dots,F_r$ being a regular sequence in $S_K$ satisfying $\deg F_i= d_i$. We know the images of $F_1,\dots, F_r$ form a $K$-basis of $I_K/\m_KI_K$. Otherwise, some $F_i$ is represented by the others in $I_K/\m_K I_K$; then the other $r-1$ polynomials will generate the ideal $I_K$ by the graded Nakayama lemma, see \cite{StacksProject_00JL}. This is contrary to the fact that $I_K=(F_1,\dots, F_r)$ and that $F_1,\dots,F_r$ is a regular sequence. In particular, for every integer $d$,
\[
\dim_K \left(I_K/\m_KI_K\right)_d
=
\#\{i\mid d_i=d\}.
\]
Since $(I/\m I)\otimes_k K\cong I_K/ \m_K I_K$, by the base-change isomorphism, we obtain
\[
\dim_k(I/\mathfrak m I)_d= \dim_K\left(I_K/\mathfrak m_KI_K\right)_d = \#\{i\mid d_i=d\}.
\]
Thus, for every degree $d$, we may choose homogeneous elements
\[
f_{d,1},\ldots,f_{d,m_d}\in I_d, \qquad m_d:=\#\{i\mid d_i=d\},
\]
whose classes form a $k$-basis of $(I/\mathfrak m I)_d$.

Collecting these elements over all degrees occurring among the $d_i$, we obtain homogeneous elements
\[
f_1,\ldots,f_r\in I,\; \text{ with }  \deg f_i=d_i,
\]
after a suitable ordering. We set
\[
J=(f_1,\ldots,f_r)\subset I.
\]
By construction, the classes of the $f_i$ form a basis of $I/\mathfrak m I$. By the graded Nakayama lemma \cite{StacksProject_00JL}, we have
\[
I=J=(f_1,\ldots,f_r),\; \text{ with } \deg f_i=d_i.
\]
This means $Z=\Proj S/I$ with $I=(f_1,\dots, f_r)$. Since $\codim(Z,\P^N)=\codim (Z_K,\P^N_K)=r$, we know $Z\subset \P^N$ is a complete intersection with degrees $(d_1,\dots,d_r)$. This completes the proof.
\end{proof}

We do some standard calculations for the tangent sheaf and normal sheaf of a reducible conic in a smooth variety. 
\begin{lem}\label{vanish implication for red conic}

Let $k$ be an algebraically closed field, let $X\subset \mathbb P^n_k$ be a
smooth projective variety, and let
\[
f:C \hookrightarrow X
\]
be the closed immersion of a reducible conic $C$ with $\O_C(1):=\O_{\mathbb P^n}(1)|_C$.

Then the following implications hold:

(i) If $H^1(C,f^*T_X\otimes\O_C(-1))=0$,
then
\[H^1(C,\NN_{C/X}\otimes\O_C(-1))=0.
\]

(ii) If $H^1(C,\NN_{C/X}\otimes\O_C(-1))=0$, then 
\[
H^1(C,f^*T_X)=0.
\]
\end{lem}

\begin{proof}
\textit{Step 1: The tangent-normal sequence for a nodal curve $C$}

First, we claim the correct tangent-normal sequence at the nodal curve $C$ is
\begin{equation}\label{normal seq with Ext}
    0\longrightarrow T_C \longrightarrow f^{*}T_X \longrightarrow
\NN_{C/X} \xrightarrow{\rho} \mathcal T_C^1 \longrightarrow 0,
\tag{A}
\end{equation}
where
\[
\mathcal T_C^1=
\Ext^1_{\O_C}
(\Omega_C,\O_C)
\cong k(s)
\]
is supported at the node $s$. Let $\I_C\subset \O_X$ be the ideal sheaf of $C$.
 The conormal sequence is right exact:
\[
\I_C/\I_C^2
\longrightarrow
f^*\Omega_X
\longrightarrow
\Omega_C
\longrightarrow 0.
\]
We claim that the first morphism is injective. We know it is injective on the smooth locus $C_{\mathrm{sm}}$, where $C\hookrightarrow X$ is an immersion of smooth varieties.
Hence its kernel is supported at the unique singular point $s$. 

We do a local calculation. We take the local chart $U=\Spec k[x,y]/(xy)$ in $C$ containing the node $s$,
and let $R=\O_{C,s}=k[x,y]_{(x,y)}/(xy)$ be the local ring. 

Consider the kernel
\[
\mathcal{K}=\ker[\I_C/\I_C^2 \to f^*\Omega_X]
\]
as a subsheaf of the conormal sheaf $\I_C/\I_C^2$. Now the coherent sheaf $\mathcal{K}$ is supported on a single point $s$.

Denote $M=\mathcal{K}_s$ as an $R$-module. If $M\neq 0$, then $\operatorname{Supp} M=\{\m_s\}$ by our
arguments above, hence $V(\operatorname{Ann}(M))=\{\m_s\}$. Since $R$ is noetherian, we have
$\m_s^N \subset \operatorname{Ann}(M)$ for some large $N$ and we obtain $\m_s^N M=0$.

Since $C\hookrightarrow X$ is a regular immersion, we know $\I_C/\I_C^2$ is locally free on $C$. 
Thus $(\I_C/\I_C^2)_s\cong R^m$ for some $m$, and we have $\mathcal{K}_s= M\subset R^m$.
Now we take $t=x+y\in \m_s$ which is a non-zero-divisor in $R$ and observe that $t^N M \subset \m_s^N M=0$, while the multiplication $a\mapsto ta$ is an injective map on $R^m$; this shows $M=0$, hence $\mathcal{K}=0$. Therefore, we obtain the short exact sequence
\[
0\longrightarrow
\I_C/\I_C^2
\longrightarrow
f^*\Omega_X
\longrightarrow
\Omega_C
\longrightarrow 0.
\]
Applying $\mathcal Hom_{\O_C}(-,\O_C)$, and using that both $\I_C/\I_C^2$ and $f^*\Omega_X$ are locally free, gives
\[
0\longrightarrow T_C
\longrightarrow f^*T_X
\longrightarrow \NN_{C/X}
\longrightarrow
\Ext^1_{\O_C}(\Omega_C,\O_C)
\longrightarrow 0.
\]
Thus
\[
\mathcal T_C^1
:=
\Ext^1_{\O_C}(\Omega_C,\O_C)
\]
is precisely the cokernel of $f^*T_X\to \NN_{C/X}$.

It remains to compute $\mathcal T_C^1$. Away from $s$, the curve $C$ is smooth, so $\Omega_C$ is locally free and hence $\mathcal T_C^1$ vanishes.
At the node, recall that $R=\O_{C,s} \cong k[x,y]_{(x,y)}/(xy)$.

The module of K\"{a}hler differentials $\Omega_{R/k}\cong \frac{Rdx \oplus Rdy}{xdy+ydx}$ has a free resolution
\[
0\longrightarrow R
\xrightarrow{\,1\mapsto (y,x)\,}
R\,dx\oplus R\,dy
\longrightarrow
\Omega_{R/k}
\longrightarrow 0.
\]
The first map is injective: if $a\in R$ satisfies $ax=ay=0$, then
$a\in (x)\cap(y)=0$.

Applying $\operatorname{Hom}_R(-,R)$ therefore yields
\[
R^{\oplus 2}
\xrightarrow{\,(a,b)\mapsto ay+bx\,}
R
\longrightarrow
\operatorname{Ext}^1_R(\Omega_{R/k},R)
\longrightarrow 0.
\]
The image of the first map is the maximal ideal $(x,y)$. Hence
\[
\operatorname{Ext}^1_R(\Omega_{R/k},R)
\cong
R/(x,y)
\cong k(s).
\]
Since $\mathcal T_C^1$ vanishes away from $s$, this proves $\mathcal T_C^1\cong k(s)$.

\medskip
\noindent

\textit{Step 2: Vanishing of $H^1(C,\NN_{C/X}\otimes \O_C(-1))$.}
We suppose that 
\[
H^1(C,f^*T_X\otimes\O_C(-1))=0.
\]
Reading the tangent-normal sequence \eqref{normal seq with Ext} for nodal curve $C$, we have
\begin{equation}\label{tangent normal TC1}
    0\longrightarrow T_C \longrightarrow f^{*}T_X \longrightarrow
\NN_{C/X} \xrightarrow{\rho} \mathcal T_C^1 \longrightarrow 0,
\tag{B}
\end{equation}
with $\mathcal T_C^1\cong k(s)$. Set
\[
K=\ker (\NN_{C/X}\xrightarrow{\rho}\mathcal T_C^1).
\]
The sequence \eqref{tangent normal TC1} induces
\[
0 \longrightarrow T_C \longrightarrow f^{*}T_X 
\longrightarrow K \longrightarrow 0.
\]
Since $C$ is a curve, tensoring the sequence above with the line bundle $\O_C(-1)$, we have
\[
H^1(C,f^*T_X\otimes \O_C (-1))=0\Longrightarrow H^1(C,K\otimes \O_C(-1))=0.
\]
Tensoring the following exact sequence with the line bundle $\O_C(-1)$
\[0\longrightarrow K \longrightarrow \NN_{C/X} \longrightarrow 
\mathcal T_C^1 \longrightarrow 0
\]
therefore gives the second arrow in the following implication.
\begin{equation}\label{Vanish H1 T_X}
H^1(C,f^*T_X \otimes \O_C(-1))=0\Longrightarrow H^1(C,K\otimes \O_C(-1))=0\Longrightarrow H^1(C,\NN_{C/X}\otimes \O_C(-1))=0.
\tag{C}
\end{equation}

\medskip
\textit{Step 3: Vanishing of $H^1(C,f^*T_X)$.}

Now we suppose that $H^1(C,\NN_{C/X}(-1))=0$. Let
\[
\nu:\widetilde C=L\amalg L'\longrightarrow C
\]
be the normalization, and let $s_L\in L$ and $s_{L'}\in L'$ be the two points lying over the node $s$.

Consider the exact sequence associated to this normalization (see \cite{StacksProject_0C44})
\[
    0\longrightarrow \O_C \longrightarrow \nu_*(\O_L \oplus \O_{L'}) 
    \xrightarrow{f(s_L)-f'(s_{L'})} k(s)\longrightarrow 0.
\]
Since $k$ is algebraically closed, we can identify $C=\Proj k[u,v,w]/(uv)$ with affine chart $U=\Spec R$ with $R=k[x,y]/(xy)$ containing the node $s=(0,0)$.

We can compute the tangent sheaf 
\[
T_U= \operatorname{Der}(R,R)= (xR) \partial_x \oplus (yR) \partial_y.
\]
For each $D\in \operatorname{Der}(R,R)$, we have $D(xy)=xD(y)+yD(x)=0$ in $R$, which forces $D(x)\in (x)$ and $D(y)\in (y)$. This means the vector fields over $C$ are those vector fields over $L\coprod L'$ vanishing at the points over the node. Canonically, we have the following isomorphism of tangent sheaves
(also see \cite[Lemma 27.6]{Hartshorne2010DeformationTheory})
\[
    T_C\cong\nu_{*}(T_L(-s_L)\oplus T_{L'}(-s_{L'})).
\]
Since $T_L\cong\O_L(2)$ and $T_{L'}\cong\O_{L'}(2)$, we obtain
\begin{equation}\label{TC vanish for H1}
H^1(C,T_C)=0.
\tag{D}
\end{equation}

Because the reducible conic $C$ is a local complete intersection over $k$ and $X$ is smooth, we know $C\to X$ is a regular immersion and $\NN_{C/X}$ is locally free by \cite[Lemma 31.23.10]{StacksProject_063P}.

Consider the double point $D_s=V(y-x)\subset \Spec k[x,y]/(xy)\subset C$ as a closed subscheme. Let $\I_{D_s}$ be the ideal sheaf of $D_s$ in $C$.

Tensoring the sequence associated to the closed immersion $D_s\to C$ with the locally free sheaf $\NN_{C/X}$, we find
\begin{equation}\label{normal seq of Ds}
0\longrightarrow \NN_{C/X}\otimes \I_{D_s} \longrightarrow \NN_{C/X}
\longrightarrow \NN_{C/X}\otimes i_*\O_{D_s} \longrightarrow 0.
\tag{E}
\end{equation}
We can take the spanning plane $\Pi$ of the reducible conic in $\P^n$. Then there is a line $L_s$ that cuts $C$ along the double point $D_s$. Fix $\O_C(1)=\O_{\P^n}(1)|_C$; we have
\[
\O_C(-1)\cong \I_{D_s}.
\]
Now we know
\[
 H^1(C, \NN_{C/X}\otimes \O_C(-1))\cong H^1(C, \NN_{C/X}\otimes \I_{D_s})=0
 \]
by assumption. Taking the cohomology sequence of \eqref{normal seq of Ds}, we find
\begin{equation}\label{H1 NCX vanish}
    H^1(C,\NN_{C/X})=0
    \tag{F}
\end{equation}
and $H^0(C,\NN_{C/X})\to H^0(D_s, \NN_{C/X}|_{D_s})$ is surjective. Since $\{s\}\to D_s$ is a closed immersion of a $0$-dimensional scheme and $\NN_{C/X}|_{D_s}$ is locally free, we know $H^0(D_s, \NN_{C/X}|_{D_s})\to \NN_{C/X}|_s$ is surjective. Hence, the evaluation
\begin{equation}\label{surj at s}
    H^0(C,\NN_{C/X})\twoheadrightarrow \NN_{C/X}|_s
    \tag{G}
\end{equation}
is surjective.

The morphism of sheaves $\rho:\NN_{C/X}\longrightarrow\mathcal T_C^1$ is surjective by \eqref{normal seq with Ext}, and since $\mathcal T_C^1$ is supported at $s$ with length $1$, the surjectivity
of \eqref{surj at s} implies that
\[
H^0(C,\NN_{C/X})
\longrightarrow \NN_{C/X}|_s \longrightarrow H^0(C,\mathcal T_C^1)
\]
is surjective. Set
\[
K=\ker (\NN_{C/X}\xrightarrow{\rho}\mathcal T_C^1).
\]
Together with \eqref{H1 NCX vanish}, the exact sequence
\[0\longrightarrow K \longrightarrow \NN_{C/X} \longrightarrow 
\mathcal T_C^1 \longrightarrow 0
\]
therefore gives
\begin{equation}\label{Vanish H1 K}
H^1(C,K)=0.
\tag{H}
\end{equation}
Finally, the sequence \eqref{normal seq with Ext} induces
\[
0 \longrightarrow T_C \longrightarrow f^{*}T_X 
\longrightarrow K \longrightarrow 0.
\]
Combining \eqref{TC vanish for H1} and \eqref{Vanish H1 K}, we conclude that
\begin{equation}\label{1-free C unobs}
H^1(C,f^{*}T_X)=0.
\tag{I}
\end{equation}

\end{proof}

\printbibliography[title=References]

\end{document}